\documentclass[preprint,12pt]{elsarticle}

\usepackage{amssymb}
\usepackage{amsmath}
\usepackage{xcolor}
\usepackage[a4paper, total={7in, 8in}]{geometry}

\usepackage{cancel}
\usepackage[section]{placeins}
\usepackage{pgfplots}
\usepackage{amsmath}
\usepackage{mathrsfs}
\usetikzlibrary{arrows}
\usepackage{accents}
\usepackage{amsmath, amsthm, amsfonts, amssymb, stmaryrd}
\usepackage[colorlinks = true, urlcolor = blue]{hyperref}
\usepackage{pst-node}
\usepackage[extra]{tipa}
\usepackage{pstricks}
\usepackage{tikz}
\usepackage[spanish,USenglish]{babel}
\usepackage[latin1]{inputenc}
\usepackage[T1]{fontenc}
\usepackage{graphicx}
\usepackage{url}
\usepackage{subfigure}
\usepackage{comment}
\usepackage{tikz}
\usepackage{amsmath}
\usetikzlibrary{matrix, arrows.meta, positioning}

\usepackage{colortbl} 
\usepackage{soul}
\usepackage{enumerate}
\usetikzlibrary{shapes.geometric}
\usepackage{graphicx}
\usepackage[ruled,vlined,linesnumbered]{algorithm2e}
\DontPrintSemicolon            
\SetKwInput{KwIn}{Input}       
\SetKwInput{KwOut}{Output}     
\newcommand{\reduline}[1]{\setulcolor{red}\ul{#1}}

\hypersetup{colorlinks=true, linkcolor=blue}

\newtheorem{thm}{Theorem}[section]

\newtheorem{cor}[thm]{Corollary}
\newtheorem{lem}[thm]{Lemma}
\newtheorem{prop}[thm]{Proposition}
\theoremstyle{definition}
\newtheorem{notn}[thm]{Notation}
\newtheorem{defn}[thm]{Definition}
\theoremstyle{remark}
\newtheorem{rem}[thm]{Remark}

\begin{document}

\begin{frontmatter}

\title{
Extremal degree-based indices of general polyomino chains via dynamic programming}

\author[]{Manuel Montes-y-Morales}
\author[]{Sayl\'e Sigarreta}
\author[]{Hugo Cruz-Su\'arez}

\address{Facultad de Ciencias F\'isico Matem\'aticas, Benem\'erita Universidad Aut\'onoma de Puebla,\\
Ave. San Claudio y R\'io Verde, Col. San Manuel, Ciudad Universitaria,\\
Puebla, Puebla 72570, M\'exico}

\begin{abstract}
In this paper, we develop a dynamic programming framework for identifying extremal general polyomino chains with respect to degree-based topological indices. As a concrete application, we resolve an open problem posed in 2015 by determining, for any given number of squares, the general polyomino chains that maximize the generalized Randi\'c index with parameter $\alpha=-1$. We show that the extremal configurations depend explicitly on the residue class of the number of squares modulo 4. Beyond this specific result, the proposed dynamic programming approach provides a constructive and systematic methodology for tackling extremal problems in graph theory.
\end{abstract}

\begin{keyword}
Polyomino chains, Generalized Randi\'c index, Extremal graph theory, Degree-based indices, Dynamic programming.

\end{keyword}

\end{frontmatter}



\section{Introduction}\label{s1}

Chemical graph theory provides a fruitful interface between discrete mathematics and chemistry, in which a molecule is modeled by a graph whose vertices represent atoms and edges represent chemical bonds. Quantitative structure-property relationships are often captured through topological indices, i.e., numerical descriptors extracted from the underlying molecular graph \cite{a8}. Among the many families of such descriptors, degree-based indices stand out due to their conceptual simplicity and their strong performance in applications \cite{a6}. A widely used example is the generalized Randi\'c index, which has been reported to correlate with physicochemical properties such as the solubility of alkanes in water \cite{a10}.\\

A natural testing ground for extremal problems involving degree-based indices is provided by graph families with a clear combinatorial structure. In this direction, polyomino systems form a classical class of planar graphs obtained by gluing unit squares edge-to-edge \cite{a17}. These objects arise in several modeling contexts, including polymers, crystal lattices, and other grid-like molecular structures \cite{a13}. A polyomino system whose inner dual graph is a path is called a polyomino chain. The study of degree-based indices on polyomino chains has attracted considerable attention, including extremal results \cite{a18,a19,a7}.\\

In this work, we focus on the broader class of general polyomino chains, i.e., polyomino chains without additional growth constraints. In particular, we adopt the convention that polyomino chains refer to the restricted (directed-growth) setting, while general polyomino chains refer to chains satisfying only the polyomino-system definition and the path condition on the inner dual. Moving from the restricted to the general class considerably enlarges the configuration space and, importantly, breaks the one-to-one correspondence between chains and simple growth encodings. This creates a genuine obstacle for extremal analysis: local descriptions remain useful, but they no longer uniquely determine a globally realizable geometry.\\

Our main contribution is a dynamic programming framework that identifies extremal general polyomino chains for any degree-based index. The key idea is to encode chains by local ``actions'' whose contributions to the target index depend only on bounded local information, enabling a recursion with optimal substructure. The resulting method yields

\begin{enumerate}
    \item [(i)] explicit recurrences for the optimal value as a function of the number of squares, and
    
    \item[(ii)] a constructive algorithm that outputs at least one extremal general polyomino chain for each size.
\end{enumerate}

This provides a unified, reusable approach to extremal questions for degree-based indices on polyomino-type families.\\

As a main application of our framework, we focus on the generalized Randi\'c index $R_{\alpha}$. In particular, we analyze the case $\alpha=-1$, namely $R_{-1}$, and thereby solve a problem posed in 2015 \cite{Rada2015} by determining, for every fixed number of squares $n$, which general polyomino chains maximize $R_{-1}$. More precisely, the maximizers fall into explicit families that depend on the residue class of $n$ modulo $4$. Overall, our results provide a constructive methodology for tackling extremal problems in graph theory within recursive families.

\section{Polyomino Chains and General Polyomino Chains}\label{s2}

The analysis in this paper centers on a particular family of topological indices known as \emph{degree-based indices}. These indices are defined by a general formula of the form:

\begin{equation}\label{TI}
   TI_f(G)= \sum_{uv\in E(G)} f(d_u, d_v),
\end{equation}
where $G$ is a graph, $E(G)$ is the set of edges of $G$, $d_u$ and $d_v$ are the degrees of the vertices $u$ and $v$, respectively, and $f$ is a real-valued symmetric function, that is, $f(x, y) = f(y, x)$ for all $x, y \in \{1, 2, \dots\}$.

\begin{defn}\label{polyomino}
Let $G$ be a graph.
\begin{enumerate}
    \item[a)] A graph $G$ is a polyomino system if it is a finite, 2-connected planar graph in which every interior face (called a \emph{cell}) is a unit square of side length one.
    \item[b)] The \emph{inner dual graph} of a polyomino system $G$ is the plane graph in which each vertex represents a cell, and two vertices are adjacent if and only if their corresponding cells share a common edge.
    \item[c)]  A \emph{polyomino chain} is a polyomino system $G$ whose inner dual graph is a path. Equivalently, a polyomino chain can be constructed iteratively by adding unit squares one at a time, according to the order induced by this path.

\end{enumerate}
\end{defn}

\begin{rem}
We call \emph{restricted polyomino chains} those obtained by adding each new cell only to the right or below the last square. This is a rotated instance of a cell-growth restriction (see  \cite{DelestViennot84} for the cell growth problem viewpoint). 
To motivate the general class, we now introduce $PC_5$ (Fig.~\ref{f1}), which is not captured by the restricted construction; see also \cite{Redelmeier81} for the unrestricted enumeration viewpoint. 
\end{rem}


\begin{figure}[h!]

 \centering
\begin{tikzpicture}[scale=0.6, transform shape]

\draw (0,0) rectangle (1,1);
\draw (1,0) rectangle (2,1);
\draw (1,-1) rectangle (2,0);
\draw (1,-2) rectangle (2,-1);
\draw (0,-2) rectangle (1,-1);

\node[fill=black, draw, circle, inner sep=2pt] (n5) at (0,0) {};
\node[fill=black, draw, circle, inner sep=2pt] (n5) at (0,1) {};
\node[fill=black, draw, circle, inner sep=2pt] (n6) at (1,0) {};
\node[fill=black, draw, circle, inner sep=2pt] (n7) at (1,1) {};
\node[fill=black, draw, circle, inner sep=2pt] (n8) at (1,-1) {};
\node[fill=black, draw, circle, inner sep=2pt] (n7) at (1,-2) {};
\node[fill=black, draw, circle, inner sep=2pt] (n8) at (0,-2) {};

\node[fill=black, draw, circle, inner sep=2pt] (n9) at (0,-1) {};
\node[fill=black, draw, circle, inner sep=2pt] (n10) at (2,1) {};
\node[fill=black, draw, circle, inner sep=2pt] (n11) at (2,0) {};
\node[fill=black, draw, circle, inner sep=2pt] (n12) at (2,-1) {};
\node[fill=black, draw, circle, inner sep=2pt] (n12) at (2,-2) {};

\end{tikzpicture}
    \caption{The eight restricted five-square chains (right/below rule) are not isomorphic to $PC_5$. Moreover, for $\alpha=-1$, their generalized Randi\'c values are pairwise distinct.}
    \label{f1}
\end{figure}
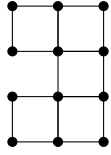

In what follows, we adopt the convention that the term \emph{polyomino chains} refers to the \emph{restricted} case, whereas the term \emph{general polyomino chains} refers to chains satisfying Definition \ref{polyomino} $c)$ without additional constraints. 

As in previous work on polyomino chains \cite{a6, a7, a0}, one can describe them as sequences of links, where each link encodes how the addition of a square changes the direction of the chain. We now extend this idea to general polyomino chains, where a more careful analysis is required: unlike the restricted case, there is no longer a unique notion of change in direction. 
To this end, we first clarify the notions of \emph{direction} and \emph{orientation} for general polyomino chains.

\begin{defn}\label{orientacion}
Let $PC_n$ be a general polyomino chain with $n$ unit squares embedded in $\mathbb{R}^2$.
\begin{enumerate}
    \item[a)] 
    The chain is \emph{horizontal} (resp.\ \emph{vertical}) at step $i\ge2$ if the $(i)$-th square is placed to the right/left (resp.\ up/down) of the $(i-1)$-th. The corresponding \emph{orientations} are encoded by the instruction set
    \begin{align*}
        \mathcal{I}=\{R,L,U,D\}\quad\text{(right, left, up, down)}.
    \end{align*}
    Although introduced for the chain, direction/orientation at step $i$ depend only on how the $(i)$-th square is attached; thus we may regard the $(i)$-th square itself as carrying them. As a convention, since all two-square chains are isomorphic, we fix $PC_2$ to be horizontal with right orientation (Fig.~\ref{f2}).
\begin{figure}[h!]

 \centering
\begin{tikzpicture}[scale=0.6, transform shape]

\draw (3,0) rectangle (4,1);

\draw (4,0) rectangle (5,1);

(n4) at (1,1) {};
\node[fill=black, draw, circle, inner sep=2pt] (n5) at (3,0) {};
\node[fill=black, draw, circle, inner sep=2pt] (n6) at (3,1) {};
\node[fill=black, draw, circle, inner sep=2pt] (n7) at (4,0) {};
\node[fill=black, draw, circle, inner sep=2pt] (n8) at (4,1) {};
\node[fill=black, draw, circle, inner sep=2pt] (n7) at (5,1) {};
\node[fill=black, draw, circle, inner sep=2pt] (n8) at (5,0) {};

\end{tikzpicture}

    \caption{The graph $PC_2$.}

    \label{f2}
\end{figure}
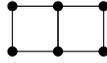
    \item[b)] 
    For $i\ge2$, the $(i+1)$-th link (governing the placement of the $(i+1)$-th square) is defined as follows:
    \begin{enumerate}\itemsep2pt
        \item[1.] \emph{Link type 1} ($L_{i+1}=1$): the $(i+1)$-th square has the \emph{same direction} as the $i$-th one (Fig.~\ref{f3}).
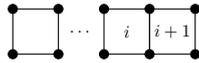
\begin{figure}[h!]

 \centering
\begin{tikzpicture}[scale=0.6, transform shape]

\draw (3,0) rectangle (4,1);

\draw (5,0) rectangle (6,1);

\draw (6,0) rectangle (7,1);

(n4) at (1,1) {};
\node[fill=black, draw, circle, inner sep=2pt] (n5) at (3,0) {};
\node[fill=black, draw, circle, inner sep=2pt] (n6) at (3,1) {};
\node[fill=black, draw, circle, inner sep=2pt] (n7) at (4,0) {};
\node[fill=black, draw, circle, inner sep=2pt] (n8) at (4,1) {};
\node[fill=black, draw, circle, inner sep=2pt] (n7) at (5,1) {};
\node[fill=black, draw, circle, inner sep=2pt] (n8) at (5,0) {};

\node[fill=black, draw, circle, inner sep=2pt] (n9) at (6,0) {};
\node[fill=black, draw, circle, inner sep=2pt] (n10) at (6,1) {};
\node[fill=black, draw, circle, inner sep=2pt] (n11) at (7,1) {};
\node[fill=black, draw, circle, inner sep=2pt] (n12) at (7,0) {};

\node at (4.5,0.5) {\small $\dots$};
\node at (6.5,0.5) {\small $i+1$};
\node at (5.5,0.5) {\small $i$};

\end{tikzpicture}

    \caption{Link type 1.}

    \label{f3}
\end{figure}
        \item[2.] Otherwise, let $j=\max\{\,1< \ell<i:\ \text{direction of square }\ell \neq\text{direction of square }i\,\}$,  whenever this maximum exists.
        \begin{itemize}
            \item If this maximum does not exist, then \emph{Link type 2} ($L_{i+1}=2$); see Fig.~\ref{f4} $a)$.
            \item If the maximum exists  and the \emph{orientation} of the $(i{+}1)$-th square matches that of the $j$-th square, then \emph{Link type 2} ($L_{i+1}=2$); see Fig.~\ref{f4} $b)$.
            \item Otherwise, \emph{Link type 3} ($L_{i+1}=3$); see Fig.~\ref{f4} $c)$.
            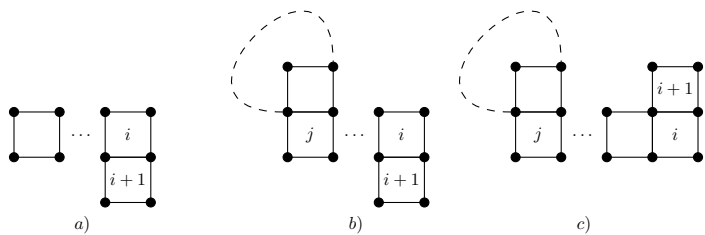
\begin{figure}[h!]
 \centering
\begin{tikzpicture}[scale=0.6, transform shape]

  \clip (-3,-2) rectangle (16,5);

\draw (-1,0) rectangle (0,1);

\draw (1,0) rectangle (2,1);

\draw (1,-1) rectangle (2,0);

\node[fill=black, draw, circle, inner sep=2pt] (n1) at (-1,0) {};
\node[fill=black, draw, circle, inner sep=2pt] (n2) at (0,0) {};
\node[fill=black, draw, circle, inner sep=2pt] (n3) at (-1,1) {};
\node[fill=black, draw, circle, inner sep=2pt] (n4) at (0,1) {};
\node[fill=black, draw, circle, inner sep=2pt] (n5) at (1,0) {};
\node[fill=black, draw, circle, inner sep=2pt] (n6) at (2,0) {};
\node[fill=black, draw, circle, inner sep=2pt] (n7) at (1,1) {};
\node[fill=black, draw, circle, inner sep=2pt] (n8) at (2,1) {};

\node[fill=black, draw, circle, inner sep=2pt] (n10) at (1,-1) {};
\node[fill=black, draw, circle, inner sep=2pt] (n11) at (2,-1) {};

\draw (5,0) rectangle (6,1);
\draw (5,1) rectangle (6,2);
\draw (7,0) rectangle (8,1);
\draw (7,-1) rectangle (8,0);

\node[fill=black, draw, circle, inner sep=2pt] (n1) at (5,0) {};
\node[fill=black, draw, circle, inner sep=2pt] (n2) at (6,0) {};
\node[fill=black, draw, circle, inner sep=2pt] (n3) at (5,1) {};
\node[fill=black, draw, circle, inner sep=2pt] (n4) at (6,1) {};
\node[fill=black, draw, circle, inner sep=2pt] (n5) at (7,0) {};
\node[fill=black, draw, circle, inner sep=2pt] (n6) at (8,0) {};
\node[fill=black, draw, circle, inner sep=2pt] (n7) at (7,1) {};
\node[fill=black, draw, circle, inner sep=2pt] (n8) at (8,1) {};
\node[fill=black, draw, circle, inner sep=2pt] (n9) at (5,2) {};
\node[fill=black, draw, circle, inner sep=2pt] (n10) at (6,2) {};
\node[fill=black, draw, circle, inner sep=2pt] (n11) at (7,-1) {};
\node[fill=black, draw, circle, inner sep=2pt] (n12) at (8,-1) {};
  \draw[ dashed] (n3) to  [out=180, in=90,looseness=5] (n10);

  \draw (10,0) rectangle (11,1);
\draw (10,1) rectangle (11,2);
\draw (12,0) rectangle (13,1);
\draw (13,0) rectangle (14,1);
\draw (13,1) rectangle (14,2);

\node[fill=black, draw, circle, inner sep=2pt] (n1) at (10,0) {};
\node[fill=black, draw, circle, inner sep=2pt] (n2) at (11,0) {};
\node[fill=black, draw, circle, inner sep=2pt] (n3) at (10,1) {};
\node[fill=black, draw, circle, inner sep=2pt] (n4) at (11,1) {};
\node[fill=black, draw, circle, inner sep=2pt] (n5) at (12,0) {};
\node[fill=black, draw, circle, inner sep=2pt] (n6) at (13,0) {};
\node[fill=black, draw, circle, inner sep=2pt] (n7) at (12,1) {};
\node[fill=black, draw, circle, inner sep=2pt] (n8) at (13,1) {};
\node[fill=black, draw, circle, inner sep=2pt] (n9) at (10,2) {};
\node[fill=black, draw, circle, inner sep=2pt] (n10) at (11,2) {};
\node[fill=black, draw, circle, inner sep=2pt] (n11) at (14,0) {};
\node[fill=black, draw, circle, inner sep=2pt] (n12) at (14,1) {};
\node[fill=black, draw, circle, inner sep=2pt] (n13) at (13,2) {};
\node[fill=black, draw, circle, inner sep=2pt] (n14) at (14,2) {};
  \draw[ dashed] (n3) to  [out=180, in=90,looseness=5] (n10);

\node at (0.5,-1.5) {\small $a)$};

\node at (1.5,-0.5) {\small $i+1$};
\node at (1.5,0.5) {\small $i$};
\node at (0.5,0.5) {\small $\dots$};
\node at (6.5,-1.5) {\small $b)$};
\node at (6.5,0.5) {\small $\dots$};
\node at (7.5,-0.5) {\small $i+1$};
\node at (7.5,0.5) {\small $i$};
\node at (5.5,0.5) {\small $j$};
\node at (11.5,-1.5) {\small $c)$};
\node at (11.5,0.5) {\small $\dots$};
\node at (13.5,1.5) {\small $i+1$};
\node at (13.5,0.5) {\small $i$};
\node at (10.5,0.5) {\small $j$};
\end{tikzpicture}
   \caption{Link types 2 and 3.}
    \label{f4}
\end{figure}
        \end{itemize}
    \end{enumerate}

    \item[c)] 
    Adopting the convention $L_1=L_2=1$, any general polyomino chain determines a sequence of the form
    $$
      (1,1,L_3,L_4,\dots,L_n),\qquad L_k\in\{1,2,3\},
    $$
    and we write the general polyomino chain as $PC(1,1,L_3,\dots,L_n)$. We call $(1,1,L_3,\dots,L_n)$ with $L_k\in\{1,2,3\}$, a \emph{sequence of links}.
\end{enumerate}
\end{defn}

\begin{defn}\label{admitidos}
Let $(1,1,L_3,\dots,L_n)$ be a sequence of links, where $L_k\in\{1,2,3\}$ for  $k\geq 3$.
Each $L_k$ is interpreted according to the placement rules for link types~1--3
(see Definition~\ref{orientacion}).
The sequence is said to be \emph{globally realizable} if the resulting construction
forms a general polyomino chain.
In this case, we denote the corresponding chain by $ PC(1,1,L_3,\dots,L_n)$. For brevity, we also refer to globally realizable sequences as \emph{valid}.

\end{defn}

\begin{rem} Note that, unlike the restricted setting with only link types $1$ and $2$, not every sequence of links is globally realizable; for example, $(1,1,1,2,1,1,3,1,1,3,1)$.
\end{rem}

In a general polyomino chain, for $ i \geq 1 $, by the definition of the links, the $ (i+1) $-th square is added to the underlying graph by attaching it to the $i$-th square along one of its sides, as specified by the $ (i+1) $-th link. Since this merging process involves only one side of the $ (i+1) $-th square, two of its vertices remain unused in the connection. We refer to these vertices as the \emph{ending vertices} of the $ (i+1) $-th square. Motivated by this,  we now state a simple lemma which will be relevant for the proof of the main theorem in Section \ref{s4}.

\begin{lem}\label{l0}
Let $(1,1,L_3,\dots,L_n)$ be a sequence of links, where $L_k\in\{1,2,3\}$ for each $k\geq 3$.
If, at each construction step, the two ending vertices of the most recently added square have degree $2$, then the sequence is globally realizable.
\end{lem}

\begin{proof}
    We argue by induction on $i$, the step of the constructive process.
    
 For \(i = 1\), the graph consisting of a single square is clearly a general polyomino chain.
Assume now that for some \(i \geq 1\), the graph formed by the first \(i\) squares is a general polyomino chain, and that the \((i+1)\)-th square is attached to this structure by identifying one of its edges with an edge of the \(i\)-th square, in such a way that its two ending vertices have degree exactly two.

To introduce notation, label the vertices of the \((i+1)\)-th square by \(v_1, v_2, v_3,\) and \(v_4\), where \(v_2\) and \(v_3\) denote its terminal vertices (see Figure~\ref{f0}).

    \begin{figure}[h!]

 \centering
\begin{tikzpicture}[scale=0.6, transform shape]


\draw (0,0) rectangle (1,1);

\node[fill=black, draw, circle, inner sep=2pt] (n1) at (0,0) {};
\node[fill=black, draw, circle, inner sep=2pt] (n2) at (1,0) {};
\node[fill=black, draw, circle, inner sep=2pt] (n3) at (0,1) {};
\node[fill=black, draw, circle, inner sep=2pt] (n4) at (1,1) {};

\node at (-0.3,-0.3) { $v_4$};
\node at (-0.3,1.3) { $v_1$};
\node at (1.3,1.3) { $v_2$};
\node at (1.3,-0.3) { $v_3$};

\end{tikzpicture}

    \caption{Labeled $(i+1)$-th square.}

    \label{f0}
\end{figure}
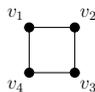

By definition, the attached $(i+1)$-th square is a cell. 
We verify that this attachment does not create any additional cells.
To this end, suppose without loss of generality that there exists a simple cycle $\mathcal{C}$ containing the vertex $v_2$. Since every vertex on a simple cycle has degree at least two and $v_2$ is adjacent only to 
$v_1$ and $v_3$, it follows that $v_1, v_3 \in \mathcal{C}$. 
Similarly, because $v_3 \in \mathcal{C}$, we must also have $v_4 \in \mathcal{C}$. Hence, any simple cycle containing either $v_2$ or $v_3$ necessarily contains the cycle $ v_1 v_2 v_3 v_4 v_1.$ Therefore, the only cell involving either $v_2$ or $v_3$ is the $(i+1)$-th square. 
Since all remaining conditions are satisfied, an inductive argument shows that the graph induced by $(1,1,L_3,\dots,L_n)$ is a general polyomino chain.

\end{proof}

A \emph{linear chain}, denoted by $Li_n$, corresponds to the case where all links are of type 1. Conversely, a \emph{zig-zag chain}, denoted by $Z_n$, is formed when all links are of type 2. In this framework, a square is called \emph{terminal} if it has exactly one adjacent square, \emph{medial} if it has two adjacent squares and contains no vertex of degree 2; and a \emph{kink} if it has two adjacent squares and contains a vertex of degree 2. A \emph{segment} is defined as a maximal linear chain together with an adjacent kink or terminal square. The length of a segment is defined as  its number of squares.

\section{General Polyomino Chains, Links and Actions} \label{s3}  

To further clarify the differences arising in the general polyomino chain setting, we begin with Lemma~2.1 of \cite{a0}, which establishes a recursive formula for the index $TI_f$ of a polyomino chain in terms of its link sequence. The argument underlying this result extends to general polyomino chains without any essential modification. We therefore state the corresponding generalized version below.
\newpage

\begin{lem}\label{l1}
Let $ TI_f $ be a degree-based index, and let $(1,1,L_3, \dots, L_n)$ be a valid sequence of links with $ n \geq 4 $. Then
$$
TI_f\big(PC(1,1,L_3, \dots, L_n)\big) = TI_f\big(PC(1,1,L_3, \dots, L_{n-1})\big) + g_f(L_{n-2},L_{n-1}, L_n)
,$$ where 
\[
g_f(L_{n-2},L_{n-1},L_n)=\]
\[
\begin{cases}
3\,f(3,3), 
& (L_{n-2},L_{n-1},L_n)=(i,1,1),\\[2pt]
3f(3,4)+f(2,4)+f(2,3)-2f(3,3), 
& (L_{n-2},L_{n-1},L_n)=(i,1,2)\ \text{or}\ (1,1,3),\\[2pt]
f(3,4)-f(2,4)+f(2,3)+2f(3,3), 
& (L_{n-2},L_{n-1},L_n)=(i,2,1)\ \text{or}\ (1,3,1),\\[2pt]
f(4,4)+2f(2,4), 
& (L_{n-2},L_{n-1},L_n)=(i,2,2)\ \text{or}\ (1,3,2),\\[2pt]
f(2,3)+f(2,4)+f(3,4)+f(4,4)-f(3,3), 
& (L_{n-2},L_{n-1},L_n)= (2,1,3)\ \text{or}\ (3,1,3),
\end{cases}
\]
with $i\in\{1,2,3\}$.
\end{lem}

The recurrence relation in Lemma~\ref{l1} shows that the contribution of each newly added square to $TI_f$ depends only on the last three links (and at least on the last two). This locality underlies the dynamic programming approach used below.

However, it is worth emphasizing that the most significant application of Lemma 2.1 to polyomino chains in \cite{a0}, arose from its use in reverse. Specifically, one could begin with an arbitrary sequence of restricted links and affirm that the corresponding value, computed via Lemma 2.1, coincides with that of the topological index $TI_f$ associated with the polyomino chain represented by the sequence. This was possible because sequences of restricted links and polyomino chains were in one-to-one correspondence. In contrast, this approach is not valid for general polyomino chains. Thus, to extend the reverse application to sequences of links and to assign them a degree-based index value in the sense of Lemma~\ref{l1}, we introduce sequences of links that, at least locally, exhibit the structural behavior of general polyomino chains. Specifically, the graph associated with any subsequence of three consecutive links corresponds to a general polyomino chain.

\begin{defn}
 A sequence of links $(1,1,L_3,\dots,L_n)$ is \emph{locally realizable} if it contains no consecutive pair $(2,3)$ or $(3,3)$.
\end{defn}

\begin{defn}\label{actions}
For $n\ge4$, triads of consecutive links are grouped according to their contribution in Lemma~\ref{l1} into the following \emph{actions}:
\[
\begin{array}{ll}
\mathrm{SS}: & (i,1,1),\\
\mathrm{SC}: & (i,1,2)\ \text{or}\ (1,1,3),\\
\mathrm{CS}: & (i,2,1)\ \text{or}\ (1,3,1),\\
\mathrm{CC}: & (i,2,2)\ \text{or}\ (1,3,2),\\
\mathrm{TT}: & (2,1,3)\ \text{or}\ (3,1,3),
\end{array}
\qquad i\in\{1,2,3\}.
\]\end{defn}

\begin{rem}
    Informally, the letter $\mathrm{S}$ stands for \emph{same direction}, $\mathrm{C}$  for \emph{change direction}, and $\mathrm{TT}$ refers to a \emph{tight turn}. Therefore, an action such as $SC$ intends to symbolize a transition from being in a same direction and then changing it (see Figs.~\ref{f3}--\ref{f4}).

\end{rem}

Given a locally realizable sequence of links $(1,1,L_3,\dots,L_n)$, we associate to each consecutive triple $(L_i,L_{i+1},L_{i+2})$ an action $a_i \in \{SS, SC, CS, CC, TT\}$, as defined in Definition~\ref{actions}. In this way, we obtain a sequence $(a_1,\dots,a_{n-2})$, which is required to satisfy the following compatibility rules:
\[
a_1\in\{SS,SC\}~
\text{and for }1\le i\le n-3:
\begin{cases}
\text{if $a_i$ ends with $S$, then } a_{i+1}\in\{SS,SC,TT\},\\
\text{if $a_i$ ends with $C$, then } a_{i+1}\in\{CS,CC\},\\
\text{if $a_i=TT$, then $a_{i-1}=CS$ and  $a_{i+1}\in\{CS,CC\}$}.
\end{cases}
\]
 
\begin{defn}
    Given $n\geq 3$, we refer to a \emph{sequence of actions} as a vector of the form $(a_1, \dots, a_{n-2})$, where $a_i \in \{SS, SC, CS, CC, TT\}$ for all $i \in \{1, \dots, n-2\}$ and the sequence satisfies the constraints established above. Moreover, if there exists a realizable \emph{sequence of links} associated to the \emph{sequence of actions}, we will say the \emph{sequence of actions} is realizable.
\end{defn}
\begin{prop}\label{action-to-links}
If $(a_1,\dots,a_{n-2})$ is a sequence of actions, then any sequence of links obtained through reverse translation is locally realizable.
\end{prop}

\begin{rem}\label{action-nonunique}
The limitation that sequences of links need not be globally realizable passes to action sequences. Also, actions retain even less geometric information and, in general, do not determine a unique representation. Figure~\ref{f5} illustrates this with two non-isomorphic, non-general chains that admit different locally realizable sequences of links but induce the same  sequence of actions.
\end{rem}

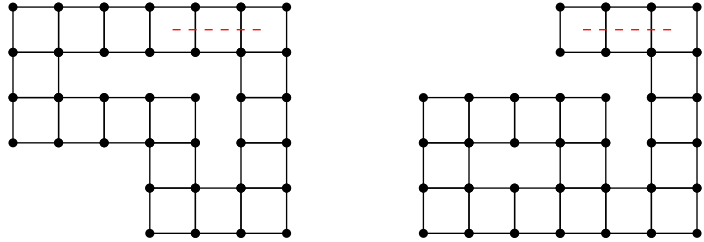
\begin{figure}[h!]
\centering
\begin{tikzpicture}[line width=0.5,scale=0.6, transform shape]

\def\squares{{0/0},{1/0},{2/0},{3/0},{4/0},
              {0/-1},{4/-1},
              {0/-2},{2/-2},{3/-2},{4/-2},
              {2/-3},{0/-3},
              {2/-4},{0/-4},
              {2/-5},{1/-5},{0/-5}}

\begin{scope}[rotate=270]
  \foreach \x/\y in \squares {
    \draw (\x,\y) rectangle (\x+1,\y+1);
  }
  \foreach \x/\y in \squares {
    \foreach \dx in {0,1} {
      \foreach \dy in {0,1} {
        \node[circle,fill=black,inner sep=2pt] at (\x+\dx,\y+\dy) {};
      }
    }
  }
\draw[red, dashed] (0.5,-1.5) -- (0.5,0.5);

\end{scope}

 \def\squares{{0/0},{1/0},{2/0},{3/0},{4/0},
              {0/-1},{4/-1},
              {0/-2},{2/-2},{3/-2},{4/-2},
              {2/-3},{4/-3},
              {2/-4},{4/-4},
              {2/-5},{3/-5},{4/-5}}

\begin{scope}[rotate=270, yshift=9cm]
  \foreach \x/\y in \squares {
    \draw (\x,\y) rectangle (\x+1,\y+1);
  }
  \foreach \x/\y in \squares {
    \foreach \dx in {0,1} {
      \foreach \dy in {0,1} {
        \node[circle,fill=black,inner sep=2pt] at (\x+\dx,\y+\dy) {};
      }
    }
  }

\draw[red, dashed] (0.5,-1.5) -- (0.5,0.5);

\end{scope}

\end{tikzpicture}

\caption{Sequences of locally realizable links for the left and right graphs are
 $(1,1,1,2,1,1,1,3,1,1,1,1,3,1,3,1,1,2)$ and $(1,1,1,2,1,1,1,3,1,3,1,2,1,1,3,1,3,1)$, respectively.
Both sequences correspond to the following sequence of actions $(SS, SC, CS, SS, SS, SC, CS, TT, CS, SC, CS, SS, SC, CS, TT, CS)$. The three initial squares are highlighted by a red dashed line.}
\label{f5}
\end{figure}

Extending the local computation of Lemma~\ref{l1} to actions, we set
\[
\begin{aligned}
&g_f(SS):=g_f(1,1,1),\quad &&g_f(SC):=g_f(1,1,2),\\
&g_f(CS):=g_f(1,2,1),\quad &&g_f(CC):=g_f(1,2,2),\\
&g_f(TT):=g_f(2,1,3),
\end{aligned}
\]
and for $n\ge4$ we have the recurrence
\begin{equation}\label{e1}
TI_f\big((a_1,\dots,a_{n-2})\big)
= TI_f\big((a_1,\dots,a_{n-3})\big) + g_f(a_{n-2}).
\end{equation}
 At this point, observe that the \emph{degree-based index value} assigned  to the \emph{sequence of actions}, according to Equation (\ref{e1}), does not necessarily 
coincide with the degree-based index value of the underlying graphs. Nevertheless, this \emph{degree-based index value} will suffice for our purposes, as will be demonstrated later.
 Henceforth, we use the term \emph{degree-based index value}  without distinction.
\begin{rem}
If a general polyomino chain $PC(1,1,L_3,\dots,L_n)$ contains no triad of the form $(2\text{ or }3,\,1,\\\,3)$, then replacing every link type~3 by type~2 yields a polyomino chain with the same degree-based index value (by Lemma~\ref{l1} and Lemma~2.1 in \cite{a0}). Hence, optimizing over that subfamily of general polyomino chains coincides with
optimizing over polyomino chains \cite{a0}.
\end{rem}

To further simplify the use of the action $TT$, we impose the restriction that
$TT$ follows $(SC,CS)$ or $(CC,CS)$ (excluding $(TT,CS)$), and we compress
\[
(\mathrm{SC},\mathrm{CS},\mathrm{TT})\mapsto \mathrm{ST},
\qquad
(\mathrm{CC},\mathrm{CS},\mathrm{TT})\mapsto \mathrm{CT}.
\]
The resulting \emph{sequence of actions} obtained via this compression is referred to as a \emph{sequence of compressed actions}. Its entries may now take values in the set $\{SS, SC, CS, CC, ST, CT\}$, while still satisfying all other previously defined constraints. 
Note that, the first entry of the \emph{sequence of compressed actions} may also be $ST$. Regarding the computation of the degree-based index, we naturally define:
\[
g_f(\mathrm{ST})=g_f(\mathrm{SC})+g_f(\mathrm{CS})+g_f(\mathrm{TT}),\qquad
g_f(\mathrm{CT})=g_f(\mathrm{CC})+g_f(\mathrm{CS})+g_f(\mathrm{TT}).
\]

\begin{rem}
   Observe that,  no valid sequence of links can contain the consecutive
subsequence $(2 \text{ or } 3,\, 1,\, 3,\, 1,\, 3),$ which corresponds to the action subsequence $(TT, CS, TT)$. Consequently,
optimizing over the set of all sequences of actions that do not
contain the subsequence $(TT, CS, TT)$, and mapping them to valid sequences of
links, is equivalent to optimizing over the entire set of valid link sequences.
\end{rem}

\section{Generating a General Polyomino Chain} \label{s4}

Based on the previous section, optimizing degree-based indices over general polyomino chains could be achieved by working with actions, as they encapsulate the essential information required for computing such indices. Therefore, we can utilize the formula of Equation (\ref{e1}) and exploit its recursive structure to apply a dynamic programming approach for identifying at least one extremal \emph{sequence of actions}. The main challenge of this approach lies in guaranteeing that the resulting extremal \emph{sequence of actions} actually corresponds to a general polyomino chain, the object of interest, in other words, that it is valid. 

\medskip
\noindent\textbf{Notation.}
Let $(a_1,\dots,a_N)$ be a sequence of \emph{compressed actions} with
$a_i\in\{SS,SC,CS,CC,ST,CT\}$. We define the cumulative number of direction changes
up to and including $a_i$ by
\[
m_0=0,\qquad
m_i = m_{i-1}+\mathbb{I}_{\{a_i\in\{SC,CC\}\}}+2\mathbb{I}_{\{a_i\in\{ST,CT\}\}},
\quad i\ge1.
\]
Each action $a_i$ consists of two letters; we denote by
$F_i\in\{S,C\}$ its first letter and by
$S_i\in\{S,C,T\}$ its second letter.
These variables will be used in the construction and validation steps below.

\begin{algorithm}[htbp]\label{A1}
\caption{Parity correction for tight turns (compressed actions)}
\KwIn{Compressed action sequence $(a_1,\dots,a_N)$}
\KwOut{Modified compressed action sequence $(a_1,\dots,a_N)$}

$m \gets 0$\;
\For{$i=1$ \KwTo $N$}{
  \If{$a_i\in\{SC,CC\}$}{$m \gets m+1$}
  \If{$a_i\in\{ST,CT\}$}{$m \gets m+2$}
  \If{$S_i=T$ \textbf{and} $m$ is odd}{
    $j \gets \max\{1,\dots,i-1: S_j=C\}$\;
    $a_j \gets F_jT$;\quad $a_i \gets F_iC$\;
  }
}
\end{algorithm}

\begin{lem}\label{l2}
    Let $(a_1,\dots,a_N)$ be a \emph{sequence of compressed actions}. For every $i\in\{1,\dots,N\}$, after completing iteration $i$ of Algorithm~\ref{A1}, the following invariants hold:
    \begin{itemize}
        \item[a)] For every $j\le i$ with $S_j=T$, $m_j$ is even.
        \item[b)] The new sequence is indeed a \emph{sequence of compressed actions}.
        \item[c)] The value of $TI_f$ is unchanged.
    \end{itemize}
\end{lem}

\begin{proof}
We argue by induction on $i$.

 $a)$ For $i=1$, if $S_1=T$ then by definition $m_1=2$, hence even.  
Assume the claim holds up to $i$, and consider iteration $i+1$.
If $S_{i+1}\neq T$ or $m_{i+1}$ is even, Algorithm~\ref{A1} does nothing and the inductive hypothesis applies.

Suppose $S_{i+1}=T$ and $m_{i+1}$ is odd. Set
$$
\ell=\max\{j\in\{1,\dots,i\}:S_j=C\},
$$
which exists because by the definition of $m_{i+1}$, we have that $m_i$ must also be odd. By maximality of $\ell$ there is no index $j\in\{\ell{+}1,\dots,i\}$ with $S_j=T$.

Algorithm~\ref{A1} updates $a_\ell: F_\ell C\mapsto F_\ell T$ and
$a_{i+1}:F_{i+1}T\mapsto F_{i+1}C$. Since $m_k$ changes parity only when passing a $C$ (increments by $+1$) and not when passing a $T$ (increments by $+2$), we have:

\begin{itemize}
    \item Before the update, the parity of $m_\ell$ equals that of $m_{i+1}$, hence $m_\ell$ is odd.
    \item After the update, the only prefix up to $\ell$ whose number of $C$-endings changes is $\ell$ itself, decreasing by $1$; thus $m_\ell$ flips to be even.
\end{itemize}

All other $m_j$ with $j\le\ell-1$ are unchanged (which are even by the inductive hypothesis when $S_j=T$),
and there is no $j\in \{\ell+1,\dots,i\}$ with $S_j=T$; indeed, if such $j$ existed,
the parity of $m_j$ would coincide with that of $m_\ell$ (no $C$ between $\ell$ and $j$) prior to the update,
contradicting the inductive hypothesis ($m_j$ even).
Finally, $S_{i+1}$ becomes $C$, so there is nothing to check at $j=i+1$.
Hence, after iteration $i+1$, every $m_j$ with $S_j=T$ and $j\le i+1$ is even.

$b)$ Notice first that, since $(a_1, \dots, a_n)$ is a sequence of compressed actions, it follows that $a_1 \in \{SS, SC, ST\}$. Hence, regardless of the modifications performed by the algorithm, if $a_1$ is equal to $ST$ or $SS$, it is never changed. If $a_1 = SC$, it may be changed to $ST$, so in all cases the first action of the new compressed sequence still lies in $\{SS, SC, ST\}$. Moreover, since in the compressed notation we move actions $TT$ together with preceding actions $SC, CS$ or $CC, CS$, in the new sequence when uncompressed it holds these precede any action $TT$. It remains to show that consecutive actions are still compatible. But for this it is sufficient to prove inductively that when a change is performed, the action following $F_lT$ starts with $C$ and that the one, if it exists, following $F_{i+1}C$ starts with $C$; but this is immediate from the fact before the change the sequence was a sequence of compressed actions and we had $F_lC$ and $F_{i+1}T$.

$c)$ By the local recurrence of Equation \eqref{e1}, $TI_f$ is the sum of local contributions $g_f(\cdot)$. Only positions $\ell$ and $i+1$ change, so it suffices to verify
\[
g_f(F_\ell C)+g_f(F_{i+1}T)=g_f(F_\ell T)+g_f(F_{i+1}C),
\]
which follows from the identity $g_f(F\,T)=g_f(F\,C)+g_f(\mathrm{CS})+g_f(\mathrm{TT})$
and additivity in Equation \eqref{e1}. Therefore, the value of $TI_f$ is unchanged.
\end{proof}

Before proceeding, note that part~ $a)$ of Lemma~\ref{l2} implicitly guarantees that Algorithm~\ref{A1} is well defined. To reach our main goal in this section, we now need a bridge from the sequence of compressed actions produced by Algorithm~\ref{A1}
to a corresponding vector of links. The Algorithm \ref{A2} provides this translation.

\begin{algorithm}[htbp]\label{A2}
\caption{Actions-to-links translation}
\KwIn{Action sequence $(b_1,\dots,b_{n-2})$}
\KwOut{Link word $(1,1,L_3,\dots,L_n)$}

$L_1 \gets 1$;\quad $L_2 \gets 1$\;
\For{$i=1$ \KwTo $n-2$}{
  \If{$b_i = TT$}{$L_{i+2} \gets 3$}
  \ElseIf{$S_i = S$}{$L_{i+2} \gets 1$}
  \Else{$L_{i+2} \gets 2$}
}
\end{algorithm}

In what follows, we assume that the first two
squares of the polyomino system lie on a horizontal line.
\begin{lem}\label{l3}
Let
$(b_1,\dots,b_{n-2})$ be the (uncompressed) sequence of actions obtained from the
compressed sequence of actions $(a_1,\dots,a_N)$ of Algorithm~\ref{A1}, and let
$(1,1,L_3,\dots,L_n)$ be the link word produced by Algorithm~\ref{A2} with input $(b_1,\dots,b_{n-2})$. Then:
\begin{itemize}
  \item[a)] For every $i\ge 5$, if $L_i=3$ then necessarily $L_{i-1}=1$ and $L_{i-2}=2$,
  and the squares $i-1$ and $i$ are horizontally aligned. In particular, if $L_i=L_{i+1}=1$ for some $i\ge1$, then $L_{i+2}\in\{1,2\}$.
  \item[b)] $(1,1,L_3,\dots,L_n)$ is a (locally realizable) sequence of links.
  \item[c)] The degree-based index is preserved:
  $TI_f\big((1,1,L_3,\dots,L_n)\big)=TI_f\big((a_1,\dots,a_N)\big)$.
\end{itemize}
\end{lem}

\begin{proof}
$a)$ For $i\ge 5$, by Algorithm~\ref{A2} we have $L_i=3$ iff $b_{i-2}=TT$. In the uncompressed action sequence, $TT$ occurs only as the last element of $(SC,CS,TT)$ or $(CC,CS,TT)$; hence $b_{i-3}=CS$ and $S_{i-4}=C$, so $L_{i-2}=2$ and $L_{i-1}=1$. Moreover, the block $(b_{i-4},b_{i-3},b_{i-2})$ ends with a $T$, and by Lemma~\ref{l2} $a)$ the number of direction changes up to
and including $b_{i-2}$ is even. Since the first two squares are horizontal and each change toggles the direction, squares $i-1$ and $i$ are horizontally aligned. To conclude, suppose for contradiction that $L_i=L_{i+1}=1$ and $L_{i+2}=3$. By the argument above, the pattern $L_{i+2}=3$ would require
$(L_i,L_{i+1})=(2,1)$, a contradiction. Hence $L_{i+2}\in\{1,2\}$.

\smallskip
$b)$ The sequence $(1,1,L_3,\dots,L_n)$ is locally admissible. Indeed, if
$(L_k,L_{k+1})=(2,3)$ then $S_{k-2}=C$ and $b_{k-1}=TT$ by Algorithm~\ref{A2};
but every $TT$ is preceded by $CS$, which forces $S_{k-2}=S$, a contradiction.
If $(L_k,L_{k+1})=(3,3)$ then $b_{k-2}=TT$ and $b_{k-1}=TT$, which is impossible
since $TT$ must be preceded by $CS$; thus consecutive $TT$ cannot occur.

\smallskip
$c)$ By Lemma~\ref{l2} $c)$, Algorithm~\ref{A1} preserves $TI_f$ on compressed actions. It remains to check that the action sequence induced from $(1,1,L_3,\dots,L_n)$ by length-three windows, say $(c_1,\dots,c_{n-2})$, agrees with the (uncompressed) input $(b_1,\dots,b_{n-2})$ to Algorithm~\ref{A2}, hence that the degree-based index computed on links matches the one on actions.

For $i=1$, we have $c_1=SS$ exactly when $L_3=1$ and $c_1=SC$ exactly when $L_3=2$,
so $c_1=b_1$. For $i\ge2$ the determination is uniform:
\[
L_{i+2}=
\begin{cases}
1 & \text{if } S_i=S,\\
2 & \text{if } S_i=C,\\
3 & \text{if } b_i=TT,
\end{cases}
\qquad
L_{i+1} =
\begin{cases}
1, & \text{if } S_{i-1}=S,\\
2, & \text{if } S_{i-1}=C,\\
3, & \text{if } S_{i-1}=T \ \ (\text{equivalently } b_{i-1}=TT).
\end{cases}
\]
and, by part $a)$, whenever $L_{i+2}=3$ one has $(L_i,L_{i+1})=(2,1)$. Therefore, the induced triad $(L_i,L_{i+1},L_{i+2})$ encodes:
\[
\begin{array}{|c|c|}
\hline
\text{pattern of } (S_{i-1},S_i)\text{ and } b_i & c_i \\ \hline
(S,S)\text{ with } b_i\neq TT & SS\\
(S,C) & SC\\
(C\text{ or }T,\,S) & CS\\
(C\text{ or }T,\,C) & CC\\
\text{any with }b_i=TT & TT\\
\hline
\end{array}
\]
which coincides with $b_i$ in all cases. Hence
$TI_f\big((1,1,L_3,\dots,L_n)\big)=TI_f\big((b_1,\dots,b_{n-2})\big)$, and the latter equals the value for the compressed output of Algorithm~\ref{A1} by Lemma~\ref{l2} $c)$.
\end{proof}

Since the variable $m_i$ was central for Algorithm~\ref{A1}, it is natural to introduce its analogue on the link side. Given the \emph{sequence of links} $(1,1,L_3,\dots,L_n)$  obtained by applying Algorithms \ref{A1} and \ref{A2}, define for $i\in\{1,\dots,n\}$
$$
M_0:=0,\qquad
M_i:=M_{i-1}+\mathbf{1}_{\{L_i\in\{2,3\}\}}.
$$
Thus $M_i$ counts the number of direction changes up to the $i$-th square.

\begin{cor}\label{c2}
Let $(1,1,L_3,\dots,L_n)$ be the sequence of links generated by Algorithms~\ref{A1}--\ref{A2}.
Then, for every $i\ge1$:
\begin{itemize}
  \item[a)] If $M_i$ is even, the $i$-th square is \emph{horizontal}. Moreover,
  if $L_i=3$ then necessarily $(L_{i-2},L_{i-1})=(2,1)$ and $M_{i-3}$ is even.
  \item[b)] If $M_i$ is odd, the $i$-th square is \emph{vertical}. In this case,
  letting $j^\star:=\max\{j\in\{2,\dots,i\}:L_j\ne1\}$, one has $L_{j^\star}=2$.
\end{itemize}
\end{cor}

To ensure realizability, we translate links into spatial instructions $R,L,U,D$ (right, left, up, down), as introduced in Section \ref{s2}.  We also use the opposites $\bar R=L$, $\bar L=R,$ $\bar U=D,$ $\bar D=U$. Intuitively, $L_i=1$ keeps the current direction and orientation; $L_i=2$ switches the direction and reuses the last orientation of that direction; and $L_i=3$ switches the direction and flips the last orientation of that direction.

\begin{algorithm}[htbp]\label{A3}
\caption{Links $\to$ instructions}
\KwIn{Link sequence $(1,1,L_3,\dots,L_n)$}
\KwOut{Instruction sequence $(I_1,\dots,I_{n-1})$}

$j \gets \max\{\,t\in\{2,\dots,n\}:~L_1=\cdots=L_t=1\,\}$\;
\lFor{$k=1$ \KwTo $j-1$}{$I_k \gets R$}
\If{$j<n$}{$I_j \gets D$}

\For{$i=j+2$ \KwTo $n$}{
  $\ell \gets \max\{\,2\le \lambda \le i-1:~L_\lambda\ne 1\,\}$\;
  \lIf{$L_i=1$}{$I_{i-1}\gets I_{i-2}$}
  \lElseIf{$L_i=2$}{$I_{i-1}\gets I_{\ell-2}$}
  \lElse{$I_{i-1}\gets \overline{I_{\ell-2}}$}
}
\end{algorithm}

Since we do not work with an arbitrary link sequence, it is useful to record the specific properties that Algorithm~\ref{A3} enforces on the instructions derived from $(1,1,L_3,\dots,L_n)$ produced by Algorithms~\ref{A1}--\ref{A2}.

\begin{lem}\label{l4}
Let $(1,1,L_3,\dots,L_n)$ be the link sequence produced by Algorithms~\ref{A1}--\ref{A2},
and let $(I_1,\dots,I_{n-1})$ be the instruction sequence obtained via Algorithm~\ref{A3}.
Then:
\begin{itemize}
  \item[a)] For every $i\in\{1,\dots,n-1\}$ one has $I_i\in\{R,L,D\}$; in particular, $U$ never occurs.
  \item[b)] None of the following \emph{contiguous patterns} appears in $(I_1,\dots,I_{n-1})$:
  $LR,$ $RL,$ $LDR,$ $RDL.$
\end{itemize}
\end{lem}

\begin{proof}
$a)$ We argue by induction on $i$.
By Algorithm~\ref{A3}, $I_1=R\in\{R,L,D\}$, so the base case holds.
Assume $I_k\in\{R,L,D\}$ for all $k\le i$ and consider $I_{i+1}$.
There are three possibilities:

\begin{enumerate}
    \item [1.] If $L_{i+2}=1$, then $I_{i+1}=I_{i}\in\{R,L,D\}$ by the induction hypothesis.
    \item[2.] If $L_{i+2}=2$, Algorithm~\ref{A3} sets $I_{i+1}=I_{\ell-2}$ where $\ell=\max\{2\le j\le i+1: L_j\ne1\}$. Hence $I_{i+1}$ is one of the previously constructed instructions and belongs to $\{R,L,D\}$.
    \item[3.] If $L_{i+2}=3$, Lemma~\ref{l3} $a)$ yields the pattern $(L_{i},L_{i+1},L_{i+2})=(2,1,3)$, so the last non-$1$ before $i+2$ occurs at $\ell=i$ and Algorithm~\ref{A3} gives $I_{i+1}=\overline{I_{\ell-2}}=\overline{I_{i-2}}$. By Corollary~\ref{c2} $a)$ applied to $i+2$, both $M_{i+2}$ and $M_{i-1}$ are even; hence the $(i-1)$-th square is horizontal and therefore $I_{i-2}\in\{R,L\}$. Since the opposite of a horizontal instruction is again horizontal, $I_{i+1}\in\{R,L\}$.
\end{enumerate}
In all cases $I_{i+1}\in\{R,L,D\}$, completing the induction and part $a)$.

\medskip
Before $b)$ we record the following mapping fact, immediate from Algorithm~\ref{A3} and Corollary~\ref{c2}: for every $i\ge1$,
\begin{equation}\label{map}
\begin{aligned}
&M_i \text{ even and } L_{i+1}=1 \ \Longrightarrow\  I_i\in\{R,L\},\\
&M_i \text{ even and } L_{i+1}=2 \ \Longrightarrow\  I_i=D,\\
&M_i \text{ odd  and } L_{i+1}=1 \ \Longrightarrow\  I_i=D,\\
&M_i \text{ odd  and } L_{i+1}\in\{2,3\}\ \Longrightarrow\  I_i\in\{R,L\}.
\end{aligned}
\end{equation}

\smallskip
$b)$ First, the pairs $(L,R)$ and $(R,L)$ cannot occur.  
Indeed, suppose $I_i\in\{R,L\}$. Then either $M_i$ is even with $L_{i+1}=1$, or $M_i$ is odd with $L_{i+1}\ne1$; in both situations $M_{i+1}$ is even. By Corollary~\ref{c2} $a)$, $L_{i+2}\in\{1,2\}$.
If $L_{i+2}=1$, Algorithm~\ref{A3} gives $I_{i+1}=I_i$;  
if $L_{i+2}=2$, then by  Equation \eqref{map} with $i\leftarrow i+1$ we have $I_{i+1}=D$.  
Thus $I_{i+1}$ is never the horizontal opposite of $I_i$, ruling out $LR$ and $RL$.

Next, the triples $(L,D,R)$ and $(R,D,L)$ cannot occur.  
Assume $I_i\in\{R,L\}$ and $I_{i+1}=D$. By Equation \eqref{map}, this forces $L_{i+2}=2$ and hence $M_{i+2}$ is odd.  
If $L_{i+3}=1$, Algorithm~\ref{A3} copies the previous instruction and $I_{i+2}=D$;  
if $L_{i+3}=2$, then the reference index in Algorithm~\ref{A3} is $\ell=i+2$, so $I_{i+2}=I_{\ell-2}=I_i$.  
In neither subcase does $I_{i+2}$ become the horizontal opposite of $I_i$, so $LDR$ and $RDL$ are excluded.

\end{proof}

We are now ready to prove the main theorem of this section, using the directional instructions as the main tool. To this end, we introduce the following \emph{safety zones} (see Figure~\ref{f6}).\\

  \emph{Right-zone($i$):} This zone consists of all squares that are either above or to the left of the $i$-th square, as well as any square that is located at least two squares above and  to the right of the $i$-th square. The $i$-th square itself is also included in this zone.\\
    
    \emph{Down-zone($i$):} This zone consists of all squares that are above the $i$-th square. The $i$-th square itself is also included in this zone.\\
    
    \emph{Left-zone($i$):} This zone consists of all squares that are either above or to the right of the $i$-th square, as well as any square that is located at least two squares above and  to the left of the $i$-th square. The $i$-th square itself is also included in this zone.\\

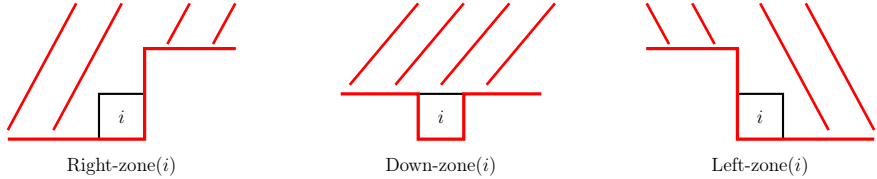
\begin{figure}[htbp]
    \centering    
\begin{tikzpicture}[line width=1pt,scale=0.6, transform shape]

    \draw[black, thick] (2,0) rectangle (3,1);
    \draw[red, very thick] (0,0) -- (3,0) -- (3,2) -- (5,2);
    \node at (2.5,0.5) {$i$};
     \node at (2.5,-0.6) {Right-zone($i$)};
    \foreach \x in {0,1}
        \draw[red] (\x,0.2) -- (\x+1.5,3);
    \foreach \x in {3.5,4.5}
        \draw[red] (\x,2.1) -- (\x+0.49,3);

    \begin{scope}[xshift=5cm]
        \draw[black, thick] (4,0) rectangle (5,1); 
        \draw[red, very thick] (2.3,1) -- (4,1) -- (4,0) -- (5,0) -- (5,1) -- (6.7,1);
        \node at (4.5,0.5) {$i$};
         \node at (4.5,-0.6) {Down-zone($i$)};
        \foreach \x in {2.5,3.5,4.5,5.5}
            \draw[red] (\x,1.2) -- (\x+1.5,3);
    \end{scope}

   \begin{scope}[xshift=19cm, xscale=-1] 
    \draw[black, thick] (2,0) rectangle (3,1);
    \draw[red, very thick] (0,0) -- (3,0) -- (3,2) -- (5,2);
  \node[xscale=-1] at (2.5,0.5) {$i$};
    \node[xscale=-1] at (2.5,-0.6) {Left-zone($i$)};

    \foreach \x in {0,1}
        \draw[red] (\x,0.2) -- (\x+1.5,3);
    \foreach \x in {3.5,4.5}
        \draw[red] (\x,2.1) -- (\x+0.49,3);
\end{scope}
\end{tikzpicture}

\caption{Graphical representation of the three different safety zones. }
\label{f6}

\end{figure}

\begin{thm}\label{t1}
Let $(1,1,L_3,\dots,L_n)$ be the link sequence obtained by applying
Algorithm~\ref{A1} followed by Algorithm~\ref{A2}. Then $(1,1,L_3,\dots,L_n)$ is a valid sequence of links (it is realizable by a general polyomino chain), and its degree-based index equals that of the compressed action sequence used to initialize Algorithm~\ref{A1}.
\end{thm}

\begin{proof}
We prove validity by induction on $i\ge2$. In view of Lemma~\ref{l0}, it suffices to show that when the $(i+1)$-th square is appended via $L_{i+1}$, the two ending vertices have degree $2$. In fact, we will maintain the invariant that the first $i$ squares lie entirely inside one of the three safety zones
Right-zone(i), Left-zone(i), or Down-zone(i) introduced in Figure~\ref{f6}, according to the instruction $I_i$ provided by Algorithm~\ref{A3}.

\medskip
\noindent\textit{Base case $i=2$.} The two initial squares are horizontal; hence they lie in Right-zone(2) and the ending vertices of the second square have degree 2.

\smallskip
\noindent\textit{Inductive step.} Assume the claim holds for some $i\ge2$: the first $i$ squares lie in the safety zone prescribed by $I_i$, and the ending vertices of the $i$-th square have degree 2. We prove the statement for $i+1$. By Lemma~\ref{l4} $a)$ , $I_i\in\{D,R,L\}$, and we split into cases.

\smallskip
\emph{Case $I_i=D$.}
By the induction hypothesis the first $i$ squares lie at or above the row of the $i$-th square. Appending the $(i+1)$-th square below the $i$-th one places it strictly below all previous squares; hence no previously placed square can meet the two lower (free) vertices of the new square, and these vertices have degree $2$. Moreover, the first $i+1$ squares lie in Down-zone($i+1$).

\smallskip
\emph{Case $I_i=R$.}
By Lemma~\ref{l4} $b)$, the pair $(L,R)$ is forbidden; hence
$I_{i-1}\in\{R,D\}$. In either subcase the induction hypothesis localizes the first $i$ squares in the corresponding safety zone.

\quad\emph{Subcase $I_{i-1}=R$.}
All first $i$ squares lie in Right-zone(i): as there are no squares among the first $i$ located below the $i$-th square or to its right and less than two squares up, it follows that  the ending  vertices of the $i+1$-th square must have degree two. Moreover, the first $i+1$ squares lie in Right-zone($i+1$). Otherwise, if there were a square outside this zone, it would contradict the assumption that the first $i$ squares lie entirely in Right-zone($i$).

\quad\emph{Subcase $I_{i-1}=D$.}
Here the first $i$ squares lie in Down-zone($i$), and the new
square is appended to the right. Thus, there are no earlier squares strictly to the right of the $(i+1)$-th square, nor directly below it. To force one of the two right (free) vertices to have degree $>2$ one would need an earlier square either immediately above the $(i{+}1)$-th square or one unite above and one unite to the right.
But Lemma~\ref{l4} $b)$ excludes the short zig-zag patterns $(L,D,R)$ and $(R,D,L)$, and combining this with the induction hypothesis (location of the first $i$ squares) rules out those two possibilities. Hence, the free vertices have degree $2$, and
again all $i+1$ squares lie in Right-zone($i+1$).

\smallskip
\emph{Case $I_i=L$.}
This is symmetric to the case $I_i=R$, using Left-zone and the
forbidden patterns in Lemma~\ref{l4} $b)$.

\medskip
This completes the induction. By Lemma~\ref{l0} the link sequence is valid. For the index identity, Lemma~\ref{l3} $c)$ shows that the translation from actions to links preserves $TI_f$, and Lemma~\ref{l2} $c)$ shows that Algorithm~\ref{A1} preserves $TI_f$ on compressed actions; chaining both equalities yields the
second claim.
\end{proof}

\section{Dynamic Programming Approach Applied to General Polyomino Chains} \label{s5}
After overcoming the main challenge of optimizing degree-based indices over general polyomino chains through the use of actions, this section focuses on the specific components of the dynamic programming framework.\\

Let $TI_f$ be a degree-based index. For $n\ge 3$, define $M_f(n,S)$ as the \emph{maximum value} of $TI_f$ over all sequences of compressed actions with $n$ squares, whose last action has the second letter $S$. Likewise, define $M_f(n,C)$ as the \emph{maximum value} of $TI_f$ over all sequences of compressed actions with $n$ squares, whose last action has the second letter in $\{C,T\}$. For the small sizes $3\le n\le 5$ (where the compressed recurrence stemming from~ Equation \eqref{e1} does not yet apply), a brute-force check yields

\begin{align*}
M_f(3,S)&=2f(2,2)+4f(2,3)+4f(3,3),\\
M_f(3,C)&=2f(2,2)+4f(2,3)+2f(3,4)+2f(2,4).
\end{align*} For $n=4,5$ one obtains

\begin{align*}
M_f(4,S)&=\max\big\{M_f(3,S)+g_f(SS), M_f(3,C)+g_f(CS)\,\big\},\\
M_f(4,C)&=\max\big\{M_f(3,S)+g_f(SC), M_f(3,C)+g_f(CC)\,\big\},\\
M_f(5,S)&=\max\big\{M_f(4,S)+g_f(SS),M_f(4,C)+g_f(CS)\,\big\},\\
M_f(5,C)&=\max\Big\{M_f(4,S)+g_f(SC),M_f(4,C)+g_f(CC),m+g_f(CS)+g_f(TT)\Big\},
\end{align*} 
 where $m=2f(2,2)+4f(2,3)+2f(2,4)+2f(3,4).$ Building on the local additivity given by Equation  \eqref{e1} and the admissibility constraints for sequences of compressed actions (Section~\ref{s3}), we obtain the following DP for $n\ge 6$:

\begin{thm}\label{t1}
Let $TI_f$ be a degree-based index, and let $M_f(n,S)$ and $M_f(n,C)$ be defined as in Section~\ref{s5}. Then, for every $n\ge 6$,
$$
M_f(n,S)=\max\big\{M_f(n-1,S)+g_f(SS), M_f(n-1,C)+g_f(CS)\big\},
$$
and
\begin{align*}
M_f(n,C)=\max\big\{&M_f(n-1,S)+g_f(SC), M_f(n-1,C)+g_f(CC),\\
                   &M_f(n-3,S)+g_f(ST), M_f(n-3,C)+g_f(CT)\big\}.
\end{align*}
\end{thm}

\begin{proof}

We prove the recurrence for $M_f(n, C)$; the case of $M_f(n, S)$ is analogous. By definition of $M_f(\cdot, \cdot)$, there exist compressed action sequences
$$
(a^1_1, \dots, a^1_{k_1}), \quad
(a^2_1, \dots, a^2_{k_2}), \quad
(a^3_1, \dots, a^3_{k_3}), \quad
(a^4_1, \dots, a^4_{k_4})
$$
with ($n-1$), ($n-1$), ($n-3$) and ($n-3$) squares, respectively, such that
$$
\begin{aligned}
&S_{k_1}=S, \qquad &&M_f(n-1,S)=TI_f((a^1_1,\dots,a^1_{k_1})),\\
&S_{k_2}\in\{C,T\}, \qquad &&M_f(n-1,C)=TI_f((a^2_1,\dots,a^2_{k_2})),\\
&S_{k_3}=S, \qquad &&M_f(n-3,S)=TI_f((a^3_1,\dots,a^3_{k_3})),\\
&S_{k_4}\in\{C,T\}, \qquad &&M_f(n-3,C)=TI_f((a^4_1,\dots,a^4_{k_4})).
\end{aligned}
$$
By Equation~\eqref{e1} and the admissibility constraints for
compressed actions (Section~\ref{s3}), appending an admissible last action yields
$$
M_f(n-1,S) + g_f(SC) = TI_f\big((a^1_1, \dots, a^1_{k_1}, SC) \big),
$$
$$
M_f(n-1,C) + g_f(CC) = TI_f\big((a^2_1, \dots, a^2_{k_2}, CC) \big),
$$
$$
M_f(n-3,S) + g_f(ST) = TI_f\big((a^3_1, \dots, a^3_{k_3}, ST) \big),
$$
$$
M_f(n-3,C) + g_f(CT) = TI_f\big((a^4_1, \dots, a^4_{k_4}, CT) \big).
$$
and each of the four right-hand sides corresponds to a compressed action sequence with $n$ squares and the last second letter in $\{C,T\}$. Since $M_f(n,C)$ is the maximum $TI_f$ among such sequences, we obtain
$$
\begin{aligned}
M_f(n,C) \ge \max\big\{&M_f(n-1,S)+g_f(SC),\ M_f(n-1,C)+g_f(CC),\\
                    &M_f(n-3,S)+g_f(ST),M_f(n-3,C)+g_f(CT)\big\}.
\end{aligned}
$$

For the reverse inequality, let $(a_1,\dots,a_k)$ be a compressed action sequence with $n$ squares, $S_k\in\{C,T\}$, and
$M_f(n,C)=TI_f((a_1,\dots,a_k))$. There are two cases.

\emph{Case $S_k=C$.} Then $a_k\in\{SC,CC\}$ and by Equation  \eqref{e1}
$$
M_f(n,C)=TI_f((a_1,\dots,a_{k-1}))+g_f(a_k)\le
\max\{M_f(n-1,S)+g_f(SC),M_f(n-1,C)+g_f(CC)\}.
$$

\emph{Case $S_k=T$.} Then $a_k\in\{ST,CT\}$, which represents the compression of a three-step block; hence, the compressed action sequence $(a_1,\dots,a_{k-1})$ has $(n-3)$
squares. Using Equation \eqref{e1},
$$
M_f(n,C)=TI_f((a_1,\dots,a_{k-1}))+g_f(a_k)\le
\max\{M_f(n-3,S)+g_f(ST), M_f(n-3,C)+g_f(CT)\}.
$$

Combining both cases yields the desired upper bound and proves the recurrence for $M_f(n,C)$.
\end{proof} By the linearity of degree-based indices (see Equation \eqref{TI}), we have
$$
TI_{-f}((a_1,\dots,a_k))=-TI_f((a_1,\dots,a_k)).
$$
Hence, minimizing $TI_f$ over all sequences of compressed actions with $n$ squares is equivalent to maximizing $TI_{-f}$ over the same class. As a direct consequence of Theorem~\ref{t1} and Equation~\eqref{e1}, we obtain a method to compute the \emph{minimum} value of any degree-based index among all compressed
action sequences with $n$ squares and a prescribed terminal second letter; denote these  minimums by $m_f(n,S)$ and $m_f(n,C)$ for $n\ge 3$.
\begin{cor}\label{c1}
Let $TI_f$ be a degree-based index. For every $n\ge 3$,
$$
m_f(n,S)=-M_{-f}(n,S),\qquad m_f(n,C)=-M_{-f}(n,C).
$$
\end{cor}

By Corollary \ref{c1}, it suffices to consider the maximization problem. Moreover, the recurrences in Theorem \ref{t1} not only yield the optimal value of $TI_f$ over all compressed action sequences with $n$ squares and a prescribed terminal second letter, but also provide a constructive procedure to obtain at least one maximizing compressed action sequence. In particular, the whole computation (value and reconstruction) runs in $O(n)$ time.

Linking back to the previous sections, Theorem \ref{t1} together with Algorithms \ref{A1} and \ref{A2} allows us to pass the obtained extremal compressed action sequence to an extremal general polyomino chain with $n$ squares.

Recalling that our main objective is to determine the maximum value of a degree-based index among all general polyomino chains with a fixed number of squares, say $ n $.  The quantity of real interest is for $n\geq 3$
$$
M_f(n) := \max\{M_f(n,S), M_f(n,C)\}.
$$
Thus, the problem reduces to analyzing $M_f(n,S)$ and $M_f(n,C)$, to which all the preceding discussion applies.

At this point, it is worth emphasizing that, by extending Proposition~3.4 from~\cite{a0}, one can recover not just a single extremal compressed
action sequence, but the complete collection of such sequences. This is
achieved by systematically tracking ties in the $\max\{\cdot\}$
expressions appearing in Theorem~\ref{t1}. By doing so, all possible
choices of actions at each recursive step are explored, allowing the
construction of every extremal action sequence attaining the value
$M_f(n,S)$ or $M_f(n,C)$. The overall procedure runs in linear time with respect to the
number of all extremal compressed action sequences, which in the
worst case grows exponentially with respect to $n$, the number of
squares.

\begin{rem}\label{r1}
  Although we are able to generate all extremal compressed action
sequences, in contrast to the situation in~\cite{a0}, which dealt
exclusively with polyomino chains, we cannot guarantee that all such
sequences are valid in the present setting. Consequently, the
combination of Theorem~\ref{t1} with Algorithms~\ref{A1} and~\ref{A2}
remains the only efficient approach for directly constructing at least an
extremal general polyomino chain. Furthermore, observe that applying
Algorithm~\ref{A1} to the entire collection of extremal
compressed action sequences would indeed transform them into valid
ones; however, this process would retain only a subset of all valid
extremal action sequences.
\end{rem}

To address the problem raised in Remark~\ref{r1}, it is necessary to
identify criteria and algorithms capable of distinguishing valid action
sequences. To this end, Algorithm~\ref{A4} describes how to generate,
from a given sequence of actions, all corresponding sequences of
instructions. The translation between actions and instructions is
established via the action-link relations introduced in
Definition~\ref{actions} and Algorithm \ref{A3}. By combining Algorithm~\ref{A4} with a
subsequent analysis of each resulting instruction sequence, carried out
by applying Lemma~\ref{l0}, we obtain a systematic procedure to
determine whether a given action sequence is valid. We refer to this
procedure as the \emph{exhaustive analysis of the action sequence}.

\begin{algorithm}[htbp]\label{A4}
\caption{A sequence of actions $\to$ all of its corresponding sequences of instructions}
\KwIn{Action sequence $(b_1,\dots,b_{n-2})$}
\KwOut{Instruction sequence $(I_1,\dots,I_{n-1})$}

$I_1 = R$; \quad $flag \gets 0$;

\For{$i = 1$ \KwTo $n-2$}{
    \lIf{$b_i = SS$ or $b_i = CS$}{$I_{i+1} \gets I_{i}$}
    \lElseIf{$b_i = CC$}{$I_{i+1} \gets I_{i-1}$}
    \lElseIf{$b_i = TT$}{$I_{i+1} \gets \overline{I_{i-2}}$}
    \lElseIf{$flag = 0$ and $b_i = SC$}{$I_{i+1} \gets D$ and $flag \gets 1$}
    \lElseIf{$b_i = SC$ and $b_{i-2} \in \{SC, CC, TT\}$}{$I_{i+1} \gets I_{i-2}$}
    \lElse{$I_{i+1} \in \{R,L,D,U\} \setminus \{I_i,\overline{I_i}\}$}
}
\end{algorithm}

\begin{rem}\label{r2}
Observe that, in Algorithm \ref{A4}, the generation of multiple instructions can occur only through actions of type $SC$. This is because $SC$ is the unique action whose corresponding triad
in Definition~\ref{actions} does not have its last entry fixed.
Consequently, in an action sequence $(b_1,\dots,b_{n-2})$, the
occurrences of $SC$ are precisely the positions at which one may choose between two distinct corresponding instructions. We therefore refer to the $SC$ actions as the \emph{pivots of the action sequence}.

Also, in Algorithm \ref{A4} we decide to fix the instruction corresponding to the first pivot found in the sequence because the graphs given by the sequences $(R, \cdots, R, D)$ and $(R, \cdots, R, U)$ are isomorphic. And we fix as well the instruction corresponding to every pivot with $b_{i-2} \in \{SC, CC, TT\}$ because choosing it to be equal to $\overline{I_{i-2}}$ would turn it into a $TT$ action, which does not correspond to the effect it should have on the $TI_f$.

\end{rem}

In light of Remark~\ref{r2}, it is straightforward to determine that the
complexity of Algorithm~\ref{A4} is \(O(2^{\text{number of pivots}})\). Moreover, since the number of pivots in a given action sequence can be typically comparable to \(n/2\) (because each use of an action $SC$ needs at least an intermediate action different from $SC$), it follows that the exhaustive analysis of an action sequence has exponential time complexity in \(n/2\). Consequently, the development of an efficient algorithm, or the identification of purely necessary or sufficient conditions that enable  the discrimination of valid action sequences, remains an open problem.

Consistent with the previous discussion, we present the following simple lemma, which will play a crucial role in the next section.

\begin{lem}\label{l5}
    Let $(b_1, \cdots, b_{n-2})$ be a sequence of actions and $l$ the index of the sequence's second pivot. Suppose $b_i$ with $i \geq l$ is a pivot and $(I_1, \cdots, I_{i})$ is a sequence of instructions obtained by applying Algorithm \ref{A4} up to action $b_{i-1}$, then

    \begin{itemize}
        \item If $b_{i-2} \in \{SC,CC,TT\}$, then the instruction for the pivot $b_i$ is fixed and equals $I_{i-2}$.
        
        \item If $I_{i-6} = I_{i-5}$, $b_{i-5} = SC, b_{i-4} = CS, b_{i-3} = TT, b_{i-2} = CS$ and  $b_{i-1} = SS$, then the instruction for the pivot $b_i$ is fixed and equals $I_{i-4}$, provided that $(I_{i-6},I_{i-5},I_{i-4},I_{i-3},I_{i-2},I_{i-1},I_{i}, I_{i+1})$ is valid.

         \item If $I_{i-7} = I_{i-6}$, $b_{i-6} = SC, b_{i-5} = CS, b_{i-4} =TT, b_{i-3} = CS, b_{i-2} = b_{i-1}= SS, b_{i+1} =CS$ and $b_{i+2} = TT$, then the instruction for the pivot $b_i$ is fixed and equals $I_{i-5}$, provided that $(I_{i-7},I_{i-6},I_{i-5},I_{i-4},I_{i-3},I_{i-2},I_{i-1},I_{i},I_{i+1},I_{i+2},I_{i+3})$ is valid.

         \item If $I_{i-4} = I_{i-3}$, $b_{i-3} = SC, b_{i-2} = CS, b_{i-1} =SS,  b_{i+1} =CS, b_{i+2} =TT$ and $b_{i+2} = CS$, then the instruction for the pivot $b_i$ is fixed and equals $I_{i-3}$, provided that $(I_{i-4},I_{i-3},I_{i-2},I_{i-1},I_{i},I_{i+1},I_{i+2},\\I_{i+3},I_{i+4})$ is valid.

    \item If $I_{i-8} = I_{i-7}$, $b_{i-7} = SC, b_{i-6} = CS, b_{i-5} =TT, b_{i-4} = CS, b_{i-3} = b_{i-2}= b_{i-1}=SS, b_{i+1} =CS, b_{i+2} =TT$ and $b_{i+3} = CS$, then the instruction for the pivot $b_i$ is fixed and equals $I_{i-6}$, provided that $(I_{i-8},I_{i-7},I_{i-6},I_{i-5},I_{i-4},I_{i-3},I_{i-2},I_{i-1},I_{i},I_{i+1},I_{i+2},I_{i+3},I_{i+4})$ is valid.

    \end{itemize}
\end{lem}

\section{Extremal General Polyomino Chains for $R_{-1}$}\label{s6}
The Randi\'c index was originally introduced by the chemist Milan Randi\'c in 1975 \cite{a23} as 
\begin{equation*}
    R(G)=\sum_{uv\in E(G) } \frac{1}{\sqrt{d_u d_v}}.
\end{equation*} This index was later generalized by Bollobás and Erdös in 1998 \cite{a24} for any $\alpha \in \mathbb{R}$ as
\begin{equation*}
    R_{\alpha}(G)=\sum_{uv\in E(G) } (d_u d_v)^{\alpha}.
\end{equation*} 
Nowadays, the generalized Randi\'c index has become one of the most extensively studied, widely applied, and well-recognized topological indices \cite{a21}. Its importance stems from its strong correlation with various chemical properties, including the boiling points, surface area, and solubility in water of alkanes \cite{a10}. $R_{-1}$, known as the second modified Zagreb index, is of particular interest due to its connection with the eigenvalues of the normalized Laplacian matrix of the graph \cite{a20}. 

Regarding extremal general polyomino chains with $n$ squares under the generalized Randi\'c index ($R_{\alpha}$), the following facts are known (see Theorems 2.1, 2.2 and 3.9 in \cite{Rada2015}): 
\begin{itemize}
    \item[a)] If $\alpha > 0$, then $Li_n$ minimizes and $Z_n$ maximizes $R_{\alpha}$.
    \item[b)] If $-0.84313 < \alpha < 0$, then $Z_n$ minimizes and $Li_n$ maximizes $R_{\alpha}$.
    \item[c)] If $-1.23853 < \alpha < -1$, then $Z_n$ minimizes and $Z_n^3$ (defined below) maximizes $R_{\alpha}$.
    \item[d)] If $-1 \leq \alpha \leq -0.84313$, then $Z_n$ minimizes $R_{\alpha}$.
\end{itemize}

In the same paper, the author raised the interesting problem of determining the extremal values of $R_{\alpha}$ over the remaining ranges of $\alpha$. Motivated by this gap, in this section we address the case $\alpha=-1$ by applying the framework developed above.

Before stating the main result for $R_{-1}$, we fix a compact notation for the canonical families of general polyomino chains that will appear in the extremal configurations.

\begin{notn}\label{n1}
Let $m \ge 0$. 
\begin{itemize}
    \item We denote by $C_{\ell_1,\ldots,\ell_r}^{\alpha_1,\ldots,\alpha_r}$ 
the family of general polyomino chains in which all vertical segments have 
length three, and the total of $\sum_{i=1}^r \alpha_i$ horizontal segments, each 
generated by a type-3 link (excluding the first one), consists of $\alpha_j$ segments of length $\ell_j$, 
for $j = 1,\ldots,r$.
  In particular:
\begin{itemize}
  \item $C_{3}^{m+1}$: exactly $m{+}1$ horizontal segments of length $3$. See Figure \ref{f10} $a)$.
  \item $C_{3,4}^{m,1}$: $m$ horizontal segments of length $3$ and one horizontal segment of length $4$.  See Figure \ref{f10} $b).$
  \item $C_{3,5}^{m,1}$: $m$ horizontal segments of length $3$ and one horizontal segment of length $5$.
  \item $C_{3,6}^{m,1}$: $m$ horizontal segments of length $3$ and one horizontal segment of length $6$.
  \item $C_{3,4}^{m-1,2}$: $m{-}1$ horizontal segments of length $3$ and two horizontal segments of length $4$.
  \item $C_{3,4}^{m-2,3}$: $m{-}2$ horizontal segments of length $3$ and three horizontal segments of length $4$.
  \item $C_{3,4,5}^{m-1,1,1}$: $m{-}1$ horizontal  segments of length $3$, one horizontal segment of length $4$, one horizontal segment of length $5$.
\end{itemize}

\item We denote by $\bar{C}_{\, 3}^{\, m+1}$ the family of general polyomino chains with $m+1$ horizontal segments and $m+1$ vertical segments, where all segments have length three. For $0 \le i \le m$, the first $i+1$ pairs of consecutive horizontal and vertical segments are generated by a type-3 link and a type-2 link, respectively (excluding the first segment), while the last $m-i$ pairs of consecutive horizontal and vertical segments are generated by a type-2 link and a type-3 link, respectively. See Figure \ref{f10} $c)$, which corresponds to $i=1.$

\item We denote by $\bar{C}_{\, 3,4}^{\,m,1}$ the family of general polyomino chains consisting of $m+1$ horizontal segments and $m+1$ vertical segments. For $0 \le i \le m$, the first $i+1$ pairs of consecutive horizontal and vertical segments are generated by a type-3 link and a type-2 link, respectively (excluding the first pair). The remaining $m-i$ pairs of consecutive horizontal and vertical segments are generated by a type-2 link and a type-3 link, respectively. All segments have length three, except for a single horizontal segment of length four, which may be either the first segment or a horizontal segment generated by a type-3 link. See Figure \ref{f10} $b),$ which corresponds to $i=m.$

\item We denote by $\underaccent{\bar}{C}_{\, 3,4}^{\, m,1}$ the family of general polyomino chains consisting of $m+1$ horizontal segments and $m+1$ vertical segments. For $0 \le i \le m$, the first $i+1$ pairs of consecutive horizontal and vertical segments are generated by a type-3 link and a type-2 link, respectively (excluding the first pair). The remaining $m-i$ pairs of consecutive horizontal and vertical segments are generated by a type-2 link and a type-3 link, respectively. All segments have length three, except for a single vertical segment of length four, which may be either the last vertical segment generated by a type-2 link or a vertical segment generated by a type-3 link. See Figure \ref{f10} $d),$ which corresponds to $i=0.$

 \item We use $Z_n^{3}$ to denote a general polyomino chain with $n$ squares, in which all horizontal segments (excluding the first one) are generated by a type-3 link, and every segment has 
length three (except when $n$ is even, in which case the last segment has length four). 
\end{itemize}
 In particular, \( Z_n^{3} = C_{3}^{m+1} \) for \( n=3+4m,\,5+4m \), while
\( Z_n^{3} \subset C_{3,4}^{m,1} \) for \( n = 4+4m \).
Note that, when \( n = 6+4m \), the final segment of \( Z_n^{3} \) is vertical. Moreover, the following inclusions hold: $C_{3}^{m+1} \subset \bar{C}_{3}^{m+1}$ for $n = 3+4m,\,5+4m, C_{3,4}^{m,1} \subset \bar{C}_{3,4}^{m,1}$ for $n = 4+4m,\,6+4m$ and $Z_n^{3} \subset \underaccent{\bar}{C}_{3,4}^{m,1}$ for $n = 6+4m$. Furthermore, each element of $\bar{C}_{\,3,4}^{\,m,1}$ is isomorphic to a unique element of $\underaccent{\bar}{C}_{\,3,4}^{\,m,1}$, see Figure \ref{f10} $b)$ and $d)$.

\end{notn}

\begin{figure}[h!]
\centering
\begin{tikzpicture}[scale=0.4, transform shape]

\begin{scope}[xshift=3cm] 
\def\squares{{0/0},{0/1},{0/2},{0/4},{1/0},{1/2},{1/4},{2/0},{2/2},{2/3},{2/4}}
\foreach \x/\y in \squares {
    \draw (\x,\y) rectangle (\x+1,\y+1);
    \foreach \dx in {0,1} {
        \foreach \dy in {0,1} {
            \node[circle,fill=black,inner sep=1.5pt] at (\x+\dx,\y+\dy) {};
        }
    }
}
\node at (1,-1) {(a)};
\end{scope}

\begin{scope}[xshift=19cm] 
\def\squares{{0/0},{0/1},{0/2},{0/4},{-2/0},{1/2},{1/4},{-1/0},{2/2},{2/3},{2/4},{-2/1},{-2/2}}
\foreach \x/\y in \squares {
    \draw (\x,\y) rectangle (\x+1,\y+1);
    \foreach \dx in {0,1} {
        \foreach \dy in {0,1} {
            \node[circle,fill=black,inner sep=1.5pt] at (\x+\dx,\y+\dy) {};
        }
    }
}
\node at (1,-1) {(c)};
\end{scope}

\begin{scope}[xshift=10cm] 
\def\squares{{0/0},{0/1},{0/2},{0/4},{1/0},{1/2},{1/4},{2/0},{2/2},{2/3},{2/4},{3/0},{3/-1},{3/-2}}
\foreach \x/\y in \squares {
    \draw (\x,\y) rectangle (\x+1,\y+1);
    \foreach \dx in {0,1} {
        \foreach \dy in {0,1} {
            \node[circle,fill=black,inner sep=1.5pt] at (\x+\dx,\y+\dy) {};
        }
    }
}
\node at (1,-1) {(b)};
\end{scope}

\begin{scope}[xshift=27cm,rotate=270,xscale=-1] 
\def\squares{{0/0},{0/1},{0/2},{0/4},{1/0},{1/2},{1/4},{2/0},{2/2},{2/3},{2/4},{3/0},{3/-1},{3/-2}}
\foreach \x/\y in \squares {
    \draw (\x,\y) rectangle (\x+1,\y+1);
    \foreach \dx in {0,1} {
        \foreach \dy in {0,1} {
            \node[circle,fill=black,inner sep=1.5pt] at (\x+\dx,\y+\dy) {};
        }
    }
}
\node[rotate=270,xscale=-1] at (-1,1.5) {(d)};
\end{scope}

rotate=270,xscale=-1

\end{tikzpicture}

\caption{According to Notation~\ref{n1}, the graphs shown in (a), (b), (c), and (d) correspond, respectively, to
$Z^{3}_{11}=C^{3}_{3}$, $C^{2,1}_{3,4}$, $\bar{C}^{3}_{3}$, and $\protect\underaccent{\bar}{C}^{\,2,1}_{3,4}$.}

\label{f10}
\end{figure}
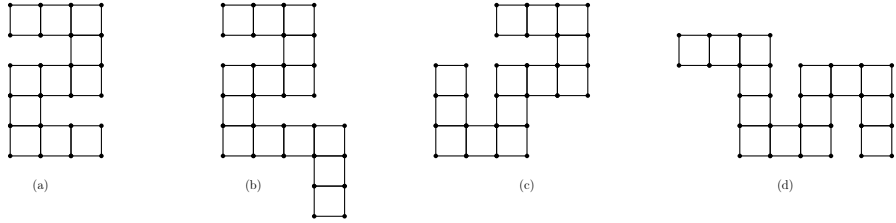

For simplicity, in this section, we will use the notation $g$, $ M(n,S) $, $ M(n,C) $, and $ M(n) $ instead of $g_f$, $ M_f(n,S), M_f(n,C)$, and $ M_f(n) $, respectively.

\begin{thm}\label{t2}
Let $n\ge 3$ and write $n=k+4m$ with $k\in\{3,4,5,6\}$ and $m\in\mathbb{Z}_{\ge0}$. Then the general polyomino chains with $n$ squares that maximize $R_{-1}$ are as follows:
\begin{itemize}
  \item[a)] When $n=3+4m$, the general polyomino chain $C_{\,3}^{\,m+1}$.
  \item[b)] When $n=4+4m$, the maximizing family is $C_{\, 3,4}^{\, m,1}$, with cardinality $m+1$.
  \item[c)] When \( n = 5 + 4m \), the maximizing families are
\( \bar{C}_{3}^{\, m+1} \), \( C_{3,5}^{\, m,1} \), and \( C_{3,4}^{\, m-1,2} \),
yielding a total of \( \frac{(m+1)(m+4)}{2} \) maximal general polyomino chains.

\item[d)] When $n=6+4m$, the maximizing families are $C_{3,6}^{\,m,1}$, $\bar{C}_{\,3,4}^{\,m,1}$, $\bar{C}_{\, 3,4}^{\, m,1}$, $C_{\, 3,4,5}^{\, m-1,1,1}$, and $C_{\, 3,4}^{\, m-2,3}$, yielding a total of \( \frac{(m+1)(m+2)(2m+18)}{12} \) maximal general polyomino chains.
\end{itemize}
Specifically, for $k\in \{3,4,5,6\}$,
$$M(k+4m)= \frac{11}{18}+\frac{1}{3}k+\frac{143}{144}m.$$

\end{thm}

\begin{proof}

To begin, carrying out base computation, we have that $g(SS)\approx 0.333$, $g(SC)\approx0.319$, $g(CS)\approx0.347$, $g(CC)\approx0.312$ , $g(ST)\approx0.993$,  $g(CT)\approx0.986$, $M(3,S) \approx 1.611$ and $M(3,C) \approx 1.583$. According to Theorem~\ref{t1}, the next step is to determine $M(3)$. 
From a direct calculation, we have that $M(3)=M(3,S).$ Thus, an extremal \emph{sequence of compressed actions} for $n=3$ is $(SS)$. For $M(4)$, 

$$
M(4,S)= M(3,S) + g(SS)   ~~ \text{and} ~~ M(4,C)= M(3,S) + g(SC).
$$

Then, it follows that, $M(4)=M(4,S)$. Thus, an extremal \emph{sequence of compressed actions} for $n=4$ is $(SS,SS)$. For $ n = 5 $, 
$$
M(5,S) = M(4,S) + g(SS)= M(4,C) + g(CS), ~~ \text{and} ~~M(5,C) = M(4,S) + g(SC).
$$ 
Therefore,  $M(5) = M(5,S)$.  Thus, taking into account the tie, extremal \emph{sequences of compressed actions} for $n=5$ are $(SS,SS,SS)$ and $(SS,SC,CS)$. At this point, observe that $2g(SS)=g(CS)+g(SC).$ Now, for $ n = 6 $, a similar argument yields:  
$$
M(6,S) = M(5,S) + g(SS) = M(5,C) + g(CS),  ~~ \text{and} ~~ M(6,C) = M(3,S) + g(ST).
$$
Moreover, when decomposing $M(6,S)$, and $M(6,C)$ into summands according to Equation (\ref{e1}),  the terms that are not shared are $g(SS)$ and $g(TT)$, respectively. We can continue this process up to $n = 10$. The essential results obtained for $n \leq 10$ are summarized in Figure~\ref{f8} as follows:

\begin{itemize}
    \item[a)] The rows represent the number of squares, while the columns correspond to the last letter of the ending action.
    \item[b)] The underlined cell $(i,j) $ represents that  $M_f(i)=M_f(i,j).$

    \item[c)] Green and blue arrows from the $(i_1,j_1)$-th entry to the $(i_2,j_2)$-th entry indicate that the extremal polyomino chains corresponding to the problem $M_f(i_1,j_1)$ are constructed using the extremal polyomino chains associated with the subproblem $M_f(i_2,j_2)$, in accordance with Theorem~\ref{t1}. As a consequence, multiple arrows emerging from the $(i,j)$-th cell indicate ties in the  \emph{max} argument of problem $M(i,j)$.

    \item[d)] The expressions in the  cells of the $i$-th row represent the non-shared summands of $M(i,S)$ and $M(i,C)$. Notice that, according to Theorem \ref{t1}, these differences fully determine the problem $M(i+1)$, eliminating the need to explicitly retain the values of $M(i,S)$ and $M(i,C)$ in the expression.

\end{itemize} According to the key information summarized in Figure~\ref{f1}, the cells display a cyclic pattern with a period of four. This pattern begins in the seventh row and extends to the tenth row.
Consequently, the recursive construction of $M(n)$, as stated in Theorem~\ref{t1}, also exhibits a cyclic structure in the decision process, again with period four. Based on the previous information, by tracking all the ties shown in Figure~\ref{f1}, we obtain the following:

\begin{itemize}
   \item When \( n = 3 + 4m \) with \( m \ge 0 \), the extremal
\emph{sequence of compressed actions} is given by
\[
(SS,\underbrace{ST,CS}_{m\text{ times}}).
\]
By recursively applying Lemma ~\ref{l5}, this sequence  corresponds to the general polyomino chain \( C^{\,m+1}_3 \).

    \item When $n = 4 + 4m$ with $m \geq 0$, the extremal
\emph{sequence of compressed actions} are given by
    \[
    (SS, \underbrace{ST, CS}_{m-i \, \text{times}}, SS, \underbrace{ST, CS}_{i \, \text{times}})
    \]
    with $0 \leq i \leq m$. By recursively applying Lemma ~\ref{l5}, these sequences correspond to  the family of general polyomino chains $C^{m,1}_{3,4}$.

    \item  When $n = 5 + 4m$ with $m \geq 0$, the extremal
\emph{sequence of compressed actions} are given by
    \[
    (SS, \underbrace{ST, CS}_{m-i \, \text{times}}, SC, CS,\underbrace{ST, CS}_{i \, \text{times}})
    \]
    and
    \[
    (SS, \underbrace{ST, CS}_{m-i-j \, \text{times}}, SS, \underbrace{ST, CS}_{j \, \text{times}}, SS, \underbrace{ST, CS}_{i \, \text{times}})
    \]
    with $0 \leq i \leq m$ and $0 \leq j \leq m-i$. By recursively applying Lemma ~\ref{l5}, these sequences correspond to  the families of general polyomino chains  $\bar{C}^{m+1}_{3}$, $C^{m,1}_{3,5}$ and $C^{m-1,2}_{3,4}$.

    \item  When $n = 6 + 4m$ with $m \geq 0$,  the extremal
\emph{sequence of compressed actions} are given by

   \[
    (SS, \underbrace{ST, CS}_{m-i-j \, \text{times}}, SC,CS, \underbrace{ST, CS}_{j \, \text{times}}, SS,\underbrace{ST, CS}_{i \, \text{times}}),
    \]
    
    \[
    (SS, \underbrace{ST, CS}_{m-i-j \, \text{times}}, SS, \underbrace{ST, CS}_{j \, \text{times}}, SC, CS,\underbrace{ST, CS}_{i \, \text{times}})
    \]
   and
    \[
    (SS, \underbrace{ST, CS}_{m-i-j-k \, \text{times}}, SS, \underbrace{ST, CS}_{k \, \text{times}}, SS, \underbrace{ST, CS}_{j \, \text{times}}, SS, \underbrace{ST, CS}_{i \, \text{times}})
    \]
    with $0 \leq i \leq m$, $0 \leq j \leq m-i$, and $0 \leq k \leq m-i-j$. By recursively applying Lemma ~\ref{l5}, these sequences correspond to   the families of general polyomino chains   $C^{m,1}_{3,6}$, $\bar{C}^{m,1}_{3,4}$, $\underaccent{\bar}{C}_{3,4}^{m,1}$, $C^{m-1,1,1}_{3,4,5}$ and $C^{m-2,3}_{3,4}$.
\end{itemize}
Finally, the concrete computations of $M(n)$ follows directly from Equation (\ref{e1}).

\end{proof}

\begin{rem}
In the proof of the previous result, once the form of all extremal sequences of compressed actions are determined for each case, the application of Algorithms \ref{A1} and \ref{A2} would yield the following. The maximum of $R_{-1}$ over general polyomino chains with $n$ squares is attained by the following families:
\begin{itemize}
  \item[a)] $n=3+4m$: the family $C_{3}^{\,m+1}$.
  \item[b)] $n=4+4m$: the family $C_{3,4}^{\,m,1}$.
  \item[c)] $n=5+4m$: the families $C_{3}^{\,m+1}$, $C_{3,5}^{\,m,1}$, and $C_{3,4}^{\,m-1,2}$.
  \item[d)] $n=6+4m$: the families $C_{3,6}^{\,m,1}$, $C_{3,4}^{\,m,1}$, $Z_n^{3}$, $C_{3,4,5}^{\,m-1,1,1}$, and $C_{3,4}^{\,m-2,3}$.
\end{itemize}

 This highlights the usefulness of Algorithms \ref{A1} and \ref{A2} as crucial tools which, once the dynamic programming analysis is completed, allow one to efficiently obtain at least one maximal structure. Even more,  for $n=3+4m,4+4m$ the above resulting families coincide with all maximal families. In particular, in this case, after completing the dynamic programming analysis, the application of the exhaustive analysis to all extremal action sequences aimed at determining all maximal structures becomes manageable thanks to the information provided by Lemma \ref{l5}, which effectively fixed all pivots in the obtained extremal action sequences and left for each only one instruction sequence to check for validity.
\end{rem}

\begin{figure}[ht]
\begin{center}
\begin{tikzpicture}[scale=0.6, transform shape, every node/.style={scale=0.5,minimum size=1cm, anchor=center},
  arrow/.style={-stealth, thick, green!70!black},
  underline/.style={-stealth, thick,red!70!black},
  every label/.style={font=\small}
]

\matrix (m) [matrix of nodes, row sep=1cm, column sep=2.5cm] {
  \textbf{$i \backslash S_{i}$} & \textbf{S} & \textbf{C}   \\
 
  \textbf{3} & \reduline{$M(3,S)$} & $M(3,C)$  \\
  \textbf{4} & \reduline{$g(SS)$} & $g(SC)$   \\
  \textbf{5} & \reduline{$g(SS)$} & $g(SC)$ \\
  \textbf{6} & \reduline{$g(SS)$} & $g(TT)$\\
  \textbf{7} & \reduline{$g(CS)$} & $g(SS)$  \\
  \textbf{8} & \reduline{$g(SS)$} & $g(SC)$ \\
  \textbf{9} &   \reduline{$g(SS)$} & $g(SC)$ \\
  \textbf{10} &  \reduline{$g(SS)$} & $g(TT)$\\
};


\draw[arrow] (m-3-2) -- (m-2-2); 
\draw[arrow] (m-3-3) -- (m-2-2); 

\draw[arrow] (m-4-2) -- (m-3-2); 
\draw[arrow] (m-4-3) -- (m-3-2); 
\draw[arrow] (m-4-2) -- (m-3-3); 

\draw[arrow] (m-5-2) -- (m-4-2); 
\draw[arrow] (m-5-2) -- (m-4-3); 
\draw[->, blue, dashed] (m-5-3) -- (m-2-2); 

\draw[arrow] (m-6-2) -- (m-5-3); 
\draw[->, blue, dashed] (m-6-3) -- (m-3-2); 

\draw[arrow] (m-7-2) -- (m-6-2); 
\draw[arrow] (m-7-2) -- (m-6-3); 
\draw[arrow] (m-7-3) -- (m-6-2); 
\draw[->, blue, dashed] (m-7-3) -- (m-4-2); 

\draw[arrow] (m-8-2) -- (m-7-2); 
\draw[arrow] (m-8-2) -- (m-7-3); 
\draw[arrow] (m-8-3) -- (m-7-2); 
\draw[->, blue, dashed] (m-8-3) -- (m-5-2); 

\draw[arrow] (m-9-2) -- (m-8-2); 
\draw[arrow] (m-9-2) -- (m-8-3); 
\draw[->, blue, dashed] (m-9-3) -- (m-6-2); 

\end{tikzpicture}
\end{center}
\caption{Graphical representation of the behavior observed in Theorem~\ref{t2}. }
\label{f8}
\end{figure}
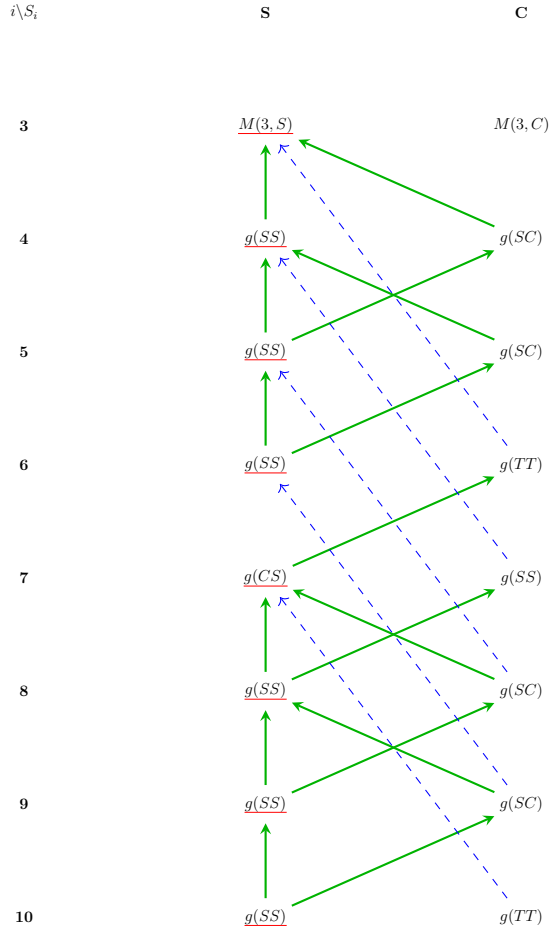

\section{Final Discussion and Code}\label{s7}

In this paper, we introduced a general framework, grounded on dynamic programming, to identify extremal general polyomino chains with respect to any degree-based index. As a concrete application, we characterized the extremal general polyomino chains that maximize the generalized Randi\'c index with parameter $\alpha=-1$. We expect that the proposed methodology will be useful for tackling related extremal problems. In particular, we leave open the problem of characterizing extremal general polyomino chains for the remaining values of the parameter. Moreover, an important open problem is the design of more efficient algorithms for identifying valid action sequences, as well as the development of additional necessary or sufficient conditions, such as Lemma~\ref{l5}; both of these directions would contribute to facilitating the identification of all extremal polyomino chains.

Finally, we provide access to our implementation of the constructive method described in Theorem~\ref{t1}, Algorithms~\ref{A1}, \ref{A2}, \ref{A3} and the exhaustive analysis of action sequences. The code (with detailed comments) efficiently computes, for any degree-based index and any prescribed number of squares $n$, a general polyomino chain attaining either the maximum or the minimum value, in linear time with respect to $n$ and the whole set of all extremal general polyomino chains with a computational complexity of $O(2^{n/2}nN)$; where $N$ is the number of extremal action sequences. The implementation is available at:
\href{https://colab.research.google.com/drive/1s4SwaMZFGjoxVJO_1nDRUNXs_6-gXy5g?usp=sharing}{\textit{Link to the Code}}.

\subsection*{Funding Information}
Sayl\'e Sigarreta was supported by CONAHCYT 2023-2024 project CBF2023-2024-1842. M. Montes-y-Morales and H. Cruz-Su\'arez received support from VIEP through grant VIEP-00544-2025.

\bibliographystyle{elsarticle-num}
\bibliography{Bib}

\end{document}